\documentclass[12pt,psfig,reqno]{amsart}
\usepackage{amssymb}
\usepackage{amsmath}
\usepackage{amsfonts}
\usepackage{amsthm}
\usepackage{enumitem}
\usepackage{graphicx}
\usepackage{extarrows}
\usepackage{color}
\usepackage{bm}
\usepackage{mathtools}
\usepackage{caption}
\usepackage{subcaption}
\usepackage{hyperref}
\usepackage{diagbox}
\usepackage{cite}
\usepackage[margin=1in]{geometry}

\newtheorem{thm}{Theorem}[section]
\newtheorem{remark}[thm]{Remark}
\newtheorem{pro}[thm]{Proposition}
\newtheorem{cor}[thm]{Corollary}
\newtheorem{exam}[thm]{Example}
\newtheorem{defn}[thm]{Definition}
\newtheorem{lemma}[thm]{Lemma}

\newtheorem{conj}[thm]{Conjecture}

\DeclareMathOperator{\diag}{diag}
\numberwithin{equation}{section}

\usepackage{color}
\begin{document}
\title{Spectral structure and eigenmatrices of self-affine measures with $p$-digits in $ \bf{\mathbb{R}^{n}} $}
\author{Jia-Long Chen$^{1}$}
\author{Wen-Hui Ai$^{2,*}$}

\address{$^1$School of Mathematics, South China University of Technology, Guangzhou, 510641, China.}
\address{$^2$Key Laboratory of Computing and Stochastic Mathematics (Ministry of Education),
	School of Mathematics and Statistics, Hunan Normal University, Changsha, Hunan 410081, P. R. China}
\email{jialongchen1@163.com(J.L. C.)}
\email{awhxyz123@163.com(W.H. A.)}
\date{\today}
\makeatletter
\@namedef{subjclassname@2020}{\textup{2020} Mathematics Subject Classification}
\makeatother
\thanks{The research is supported in part by the NNSF of China (Nos. 12201206 and 12371072), the Hunan Provincial NSF (No. 2024JJ6301). The first author was also supported by the Doctoral Student Program of the Young S$\&$T Talents Cultivation Project, CAST. \\
$^*$Corresponding author.}

\begin{abstract}
For a prime number $p > 2$, let $\bm{0} \in D \subset \mathbb{Z}^n$ be a $p$-element digit set satisfying
$
\mathcal{Z}(\widehat{\delta}_D) =\cup_{j=1}^{p-1}(\frac{j}{p}\bm{a}+\mathbb{Z}^{n})
$
for some \( \bm{a} \in \{ (i_1, \dots, i_n)^t : i_k \in [1, p-1] \cap \mathbb{Z}, 1\leq k\leq n \} \), where $\mathcal{Z}(\widehat{\delta}_D)$ is the zero set of the Fourier transform of $\delta_D$. Let $Q$ be an integer expansive diagonal matrix in $\mathbb{R}^{n}$, the self-affine measure $\mu_{Q,D}$ is defined by
\[
\mu_{Q,D}(\cdot) = \frac{1}{\#D} \sum_{d \in D} \mu_{Q,D}(Q(\cdot) - d).
\]
In this paper, we first provide sufficient condition for a maximal orthogonal family $E_{\Lambda}=\{ e^{-2\pi i \langle \lambda, x \rangle} : \lambda\in \Lambda \subset\mathbb{R}^{n}\}$ to be an orthogonal basis of $L^2(\mu_{Q,D})$ when $Q = p\operatorname{diag}[q, \dots, q]$ with $q\geq 1$.
Then we obtain necessary and sufficient conditions for the integer matrix $R$ such that $E_{\Lambda}$ and $E_{R\Lambda}$ are both orthogonal basis of $L^{2}(\mu_{Q,D})$.
Furthermore, for $Q = p\operatorname{diag}[l_1, \dots, l_n]$ with $|l_1 \dots l_n| > 1$, we give a necessary and sufficient condition under which the real diagonal matrix is the second type spectral eigenmatrices of $\mu_{Q,D}$.
\end{abstract}

\maketitle

\tableofcontents

\section{\bf{Introduction}}
Let $ \mu $ be a Borel probability measure on Euclidean space $ \mathbb{R}^{n} $, and let $\widehat{\mu}$ be the Fourier transform of $\mu$, i.e.,
 \begin{equation*}
 \widehat{\mu}(\xi)=\int_{\mathbb{R}^{n}}e^{-2\pi i  \langle\xi, x\rangle } d\mu(x),\quad \xi\in\mathbb{R}^{n} .
 \end{equation*}
A fundamental problem in harmonic analysis for $L^{2}(\mu)$ is whether there exists a
discrete set $\Lambda\subset\mathbb{R}^{n}$ such that the collection of exponential functions
\begin{equation*}
E_{\Lambda}:=\left\{e^{-2\pi i\langle\lambda, x\rangle}:\lambda\in \Lambda \right\}
\end{equation*}
forms an orthonormal basis for $L^{2}(\mu)$.
If this holds, $\mu$ is called a spectral measure
and $\Lambda$ is a spectrum of $\mu$, and $(\mu, \Lambda)$ is called a spectral pair.
The study of spectral measures was first initiated by Fuglede in
1974 \cite{B1}, who proposed the famous
{\it Fuglede's Conjecture:} $\mathcal{X}_{\Omega}(x)dx$  is a spectral measure if and only if $\Omega$ is
a translation tile $\mathbb{R}^{n}$.
This conjecture was disproven in various directions by Tao and others in $n \geq 3$ \cite{T1,KM1, FR, M}.
Nevertheless, the conjecture remains open in several important cases, particularly in low dimensions.
For comprehensive surveys and recent developments, we refer to \cite{L, LL, LWInvent, LM} and their bibliographies.

In 1998, Jorgensen and Pedersen \cite{JP1} found the first singular spectral measure, which is the middle-fourth Cantor measure. This result has greatly stimulated subsequent research. In \cite{HL}, Hu and Lau showed that $\mu_{\rho^{-1},\{0,1\}}$ admits an infinite orthonormal set if and only if $\rho$ is the $n$-th root of $p/q$, where $p$ is an odd integer and $q$ is an even integer. The characterization problem was finally resolved by Dai \cite{D1}, who proved that the aforementioned Cantor measures $\mu_{2k,\{0,1\}}$ are the only class of spectral measures among the family of measures $\mu_{\rho^{-1},\{0,1\}}$. Dai, He and Lau \cite{DHL1} also investigated $N$-Bernoulli measures $\mu_{\rho^{-1},\{0,1,\dots,N-1\}}$.

Let $ Q \in M_{n}(\mathbb{R}) $ be an $ n\times n $ expansive real matrix (i.e., all the eigenvalues of $Q$ have
modulus strictly greater than one) and $ D $ be a finite subset of $ \mathbb{R}^{n} $ with cardinality $ \#D=p $. We refer to system $ \{f_{d}(x)=Q^{-1}(x+d)\}_{d\in D} $ as an iterated function system (IFS) and there is a unique nonempty compact set $ T $ such that $ T=\cup_{d \in D}f_{d}(T) $, where $ T $ is called invariant set (or attractor)  \cite{F1,H1}. Then the IFS arises a natural self-affine measure $ \mu_{Q,D} $ satisfying
\begin{equation}\label{1.1}
	\mu_{Q,D}=\dfrac{1}{\#D}\sum_{d\in D}\mu_{Q,D}\circ f_{d}^{-1}=\frac{1}{\#D} \sum_{d \in D} \mu_{Q,D}(Q(\cdot) - d).
\end{equation}
Such a measure $ \mu_{Q,D} $ is supported on the attractor $ T $ of the IFS $ \{f_{d}\}_{d\in D} $, where
\begin{equation*}
	T:=T(Q,D)=\left\{\sum_{k=1}^{\infty}Q^{-k}d_{k}:d_{k}\in D\right\}:=\sum_{k=1}^{\infty}Q^{-k}D.
\end{equation*}
It is known that a self-affine measure $ \mu_{Q,D} $ can
be expressed by the infinite convolution of Dirac measures with equal weights as follows:
\begin{equation}\label{eq1.2}
	\mu_{Q,D}=\delta_{Q^{-1}D}*\delta_{Q^{-2}D}*\dots*\delta_{Q^{-k}D}*\dots,
\end{equation}
where $ \delta_{E}=\#E^{-1}\sum_{e\in E}\delta_{e} $ and $ \delta_{e} $ is the Dirac probability measure with mass on the point $ e $, and the convergence is in a weak sense.

In this paper, for a prime number $p > 2$, we consider the self-affine measure $\mu_{Q,D}$ given by \eqref{1.1}, which satisfies the following conditions: $Q = p\operatorname{diag}[l_1, \dots, l_n]\in pM_{n}(\mathbb{Z})$, and
\begin{equation}\label{1.3}
	\mathcal{Z}(\widehat{\delta}_D)=\left\{x| \widehat{\delta}_D(x)=0\right\} = \bigcup_{j=1}^{p-1} \left( \frac{j}{p}\bm{a} + \mathbb{Z}^{n} \right),\;\text{with}\; \#D=p,
\end{equation}
for some vector $\bm{a} \in \{ (i_1, \dots, i_n)^t : i_k \in [1, p-1] \cap \mathbb{Z},\ 1\leq k\leq n \}$. The fractal measures generated by digit sets satisfying condition \eqref{1.3} are classical, and their spectral properties have been extensively studied. The known results can be found in \cite{WLS1,CC,Y1,CYZ}, and we will briefly review them in subsequent sections. The main results of this paper fall into two parts: one concerns the spectral structure, while the other addresses the spectral eigenmatrix problem.

\subsection{\bf{Spectral structure}}

Characterizing the spectra of a given spectral measure is fundamental to investigating the diverse properties of spectral measures. This problem was first addressed by Dutkay, Han and Sun \cite{DHS}, who employed 4-adic expansions with the digit set $\{0, 1, 2, 3\}$ to characterize the maximal orthogonal sets of the middle-fourth Cantor measure $\mu_{4,\{0,2\}}$, and further mapped these sets to the labelings of a binary tree. For the Cantor measure $\mu_{\rho^{-1},\{0,1,\dots,N-1\}}$, Dai, He and Lai \cite{DHL} characterized all maximal orthogonal sets $\Lambda$ in the case where $N$ divides $\rho^{-1}$. Their approach relied on a maximal tree mapping defined on the $N$-adic tree, within which every element of $\Lambda$ admits a unique representation via finite $\rho^{-1}$-adic expansions. Later, An, Dong and He \cite{ADH} established the tree structure for the planar Sierpinski measure $\mu_{Q,D}$ generated by the expanding matrix $Q = \operatorname{diag}[3q, 3q]$ with $q \geq 1$ and the digit set $D = \{(0,0)^t, (0,1)^t, (1,0)^t\}$. As shown in the above results, the digit sets of all the measures studied satisfy condition \eqref{1.3}, which is not difficult to verify. Moreover, the spectrality of measures generated by digit sets satisfying \eqref{1.3} has been widely investigated\cite{Y1,CYZ}. In \cite{CY}, it was shown that the measure $\mu_{Q,D}$ generated by a digit set obeying \eqref{1.3} and an expanding matrix $Q\in M_{n}(\mathbb{Z})$ is a spectral measure if and only if $Q^{t}\bm{a}\in p\mathbb{Z}^{n}$. Therefore, it is natural to employ tree structures to describe its spectra. Before presenting the results in the first part, we introduce some necessary notation and concepts.

For a prime number $p > 2$, let $Q = p\operatorname{diag}[q, \dots, q]\in pM_{n}(\mathbb{Z})$, where $q \geq 1$. We consider the $Q$-adic expansion with digit set
\begin{equation*}
	\Gamma_{q}=Q\left[-\frac{1}{2},\frac{1}{2}\right)^{n}\cap \mathbb{Z}^{n}=\left\{(i_{1},i_{2},\dots,i_{n})^{t}\in\mathbb{Z}^{n}:-\dfrac{pq}{2}\le i_{1},i_{2},\dots,i_{n}<\dfrac{pq}{2}\right\},
\end{equation*}
which forms a complete residue system modulo $Q$ in $\mathbb{Z}^{n}$. Therefore, every element of $\mathbb{Z}^{n}$ has a unique $Q$-adic expansion. Furthermore, $\Gamma_q$ can be decomposed into a disjoint union:
\[
\Gamma_q = \bigcup_{\bm{b} \in B_q} \mathcal{A}_{\bm{b}},
\]
where $\mathcal{A}_{\bm{b}} \equiv \bm{b} + A_q \pmod{Q}$ is a maximal orthogonal set for $\mu_{Q,D}$ in $\Gamma_q$,
\begin{equation*}
	A_q = q\{\bm{0}, \bm{a}, \dots, (p-1)\bm{a}\},
\end{equation*}
and
\begin{equation*}
	B_q = \left\{(i_1, i_2, \dots, i_n)^t \in \mathbb{Z}^n : -\frac{q}{2} \le i_1 < \frac{q}{2},\ -\frac{pq}{2} \le i_2, \dots, i_n < \frac{pq}{2}\right\}.
\end{equation*}
We will prove this decomposition in Section 2. Let $\Sigma_{p} = \{0, 1, \dots, p-1\}$, and for $k \ge 1$, let $\Sigma_{p}^{k} = \{\mathbf{I} = i_{1}i_{2}\dots i_{k} : i_j \in \Sigma_{p} \text{ for all } j = 1,2,\dots,k\}$. We define $\Sigma_{p}^{*} = \cup_{k=0}^{\infty}\Sigma_{p}^{k}$, where $\Sigma_{p}^{0} = \{\emptyset\}$. Note that $\Sigma_{p}^{*}$ can be interpreted as a tree: the root is the empty set $\emptyset$, and for each node $\mathbf{I}$, its child nodes are $\mathbf{I}\sigma$ for all $\sigma \in \Sigma_{p}$ (or equivalently, the set of child nodes is $\mathbf{I}\Sigma_{p}$).
\begin{defn}\label{defn1.3}
	 We call a map $ \gamma: \Sigma_{p}^{*}\rightarrow \Gamma_{q} $ a maximal tree map for $ \mu_{Q,D} $ if
\begin{enumerate}
		\item[\rm(i)] $\gamma(0^{k}i_{k+1})=i_{k+1}q\bm{a}\pmod Q $ for $ k\ge0 $ and $i_{k+1}\in\Sigma_{p}$; here, $0^{0}i_{1}=\emptyset i_{1}=i_{1}$;
		\item[\rm(ii)] For $ \mathbf{I}i\in\Sigma_{p}^{*},$ $ \gamma(\mathbf{I}i)=e_{\mathbf{I}}+iq\bm{a}\pmod Q $, where $e_{\mathbf{I}}\in B_{q}; $
		\item[\rm(iii)] For any $ \mathbf{I}\in\Sigma_{p}^{*} $, there exists $ \mathbf{J}\in \Sigma_{p}^{*}  $ such that $ \gamma(\mathbf{IJ}0^{k})=\bm{0} $ for sufficiently large $ k $.
	\end{enumerate}
\end{defn}
We denote
\begin{equation*}
	\Sigma_{p}^{\gamma} = \{\mathbf{I} \in \Sigma_{p}^{\infty} : \gamma(\mathbf{I}|_{k}) = \bm{0} \text{ for all sufficiently large } k\},
\end{equation*}
where $\mathbf{I}|_{k}$ denotes the prefix of $\mathbf{I}$ of length $k$. It is known that $ \Sigma_{p}^{\gamma} \subset \Sigma_{p}^{*}0^{\infty} $. Furthermore, we can define a mapping $\gamma^{*}: \Sigma_{p}^{\gamma} \to \mathbb{Z}^{n}$ by
\begin{equation*}
	\gamma^{*}(\mathbf{I}) = \sum_{k=1}^{\infty}(pq)^{k-1}\gamma(\mathbf{I}|_{k}) \quad \text{for all } \mathbf{I} \in \Sigma_{p}^{\gamma}.
\end{equation*}
\begin{defn}
	Let $\mathbf{I} = i_{1}i_{2}\dots \in \Gamma_{q}^{*} \cup \Gamma_{q}^{\infty}$. If there exists an integer $N$ such that $i_{N} \neq 0$ and $i_{k} = 0$ for all $k > N$, then $N$ is called the effective length of $\mathbf{I}$, denoted by $\ell(\mathbf{I}) = N$. In particular, if $\mathbf{I} = 0^{k}$ or $\mathbf{I} = 0^{\infty}$, then $\ell(\mathbf{I}) = 0$.
\end{defn}
Let $\mathbf{J} = j_{1}j_{2}\dots$, and let $\gamma$ be a maximal tree map for $\mu_{Q,D}$. For any $\mathbf{J} = j_{1}j_{2}\dots \in \Sigma_{p}^{*} \cup \Sigma_{p}^{\infty}$, we define
\begin{equation*}
	\Upsilon_{\gamma}(\mathbf{J}) = \gamma(\mathbf{J}|_{1})\gamma(\mathbf{J}|_{2})\dots \in \Gamma_{q}^{*} \cup \Gamma_{q}^{\infty},
\end{equation*}
where $\Gamma_{q}^{*}$ and $\Gamma_{q}^{\infty}$ are defined analogously to $\Sigma_{p}^{*}$ and $\Sigma_{p}^{\infty}$. We define $\mathbf{J}_{n,m} = j_{n}j_{n+1}\dots j_{m-1}$ for $n < m$, and $\mathbf{J}_{n,m} = \emptyset$ if $n = m$. For any $\mathbf{I} \in \Sigma_{p}^{k}$ and $\mathbf{J} \in \Sigma_{p}^{\infty}$ such that $\mathbf{IJ} \in \Sigma_{p}^{\gamma}$, we decompose the sequence $\Upsilon_{\gamma}(\mathbf{IJ})$ as
\begin{equation}\label{eq2.4}
	\Upsilon_{\gamma}(\mathbf{IJ}) = \Upsilon_{\gamma}(\mathbf{I}) \cdot \Upsilon_{\gamma}(\mathbf{IJ})_{n_{0},n_{1}} \cdot \Upsilon_{\gamma}(\mathbf{IJ})_{n_{1},n_{2}} \dots \Upsilon_{\gamma}(\mathbf{IJ})_{n_{m},n_{m+1}},
\end{equation}
where $n_{0} = k + 1 \le n_{1} < n_{2} < \dots < n_{m+1} = \infty$, and for each $1 \le j \le m$, the last digit of $\mathbf{IJ}|_{n_{j}}$ is $0$ while $\gamma(\mathbf{IJ}|_{n_{j}}) \neq 0$.
Based on \eqref{eq2.4}, we define
\begin{equation*}
	\Theta_{\mathbf{I}}(\mathbf{J}) = \sum_{k=0}^{m}\ell(\Upsilon_{\gamma}(\mathbf{IJ})_{n_{k},n_{k+1}}),
\end{equation*}
where this quantity depends on the partition $\{n_{0}, n_{1}, \dots, n_{m}\}$. By the definition of the maximal tree map $\gamma$, it follows that $\Theta_{\mathbf{I}}(\mathbf{J}) < \infty$.

We now present our results concerning the tree structure.
\begin{thm}\label{thm1.1}
Let $Q = p\mathrm{diag}[q,\dots, q]\in pM_{n}(\mathbb{Z})$ with prime $p>2$ and integer $q\geq1$, $D\subset\mathbb{Z}^{n}$ satisfy \eqref{1.3}. Then $0\in\Lambda \subset \mathbb{R}^{n}$ is a maximal orthogonal set for $\mu_{Q,D}$ if and only if there exists a maximal tree map $\gamma$ such that $\Lambda = \gamma^{*}(\Sigma_{p}^{\gamma})$.
\end{thm}
\begin{thm}\label{thm1.2}
Let $Q = p\mathrm{diag}[q,\dots, q]\in pM_{n}(\mathbb{Z})$ with prime $p>2$ and integer $q\geq1$, $D\subset\mathbb{Z}^{n}$ satisfy \eqref{1.3}. Let $\gamma$ be a maximal tree map and $\Lambda = \gamma^{*}(\Sigma_{p}^{\gamma})$. If for each $\mathbf{I} \in \Sigma_{p}^{*}$, there exists a $\mathbf{J}_{\mathbf{I}}$ such that $\mathbf{I}\mathbf{J}_{\mathbf{I}} \in \Sigma_{p}^{\gamma}$ and $\sup_{\mathbf{I} \in \Sigma_{p}^{*}} \Theta_{\mathbf{I}}(\mathbf{J}_{\mathbf{I}}) < \infty$, then $\Lambda$ is a spectrum of $\mu_{Q,D}$.
\end{thm}
The following corollary follows immediately.
\begin{cor}\label{col1.5}
	The set
	\begin{equation*}
		\Lambda(Q,C_{q}) := \left\{ \sum_{k=0}^{n} Q^{k} C_{q} : n \ge 1 \right\}
		= C_{q} + Q C_{q} + Q^{2} C_{q} + \cdots
	\end{equation*}
	is a spectrum of $\mu_{Q,D}$, where $C_{q} \equiv A_{q} \pmod{Q}$.
\end{cor}

\medskip

\subsection{\bf{Spectral eigenmatrix}}

Let $\mu$ be a singular spectral measure in $\mathbb{R}^{n}$. Spectral eigenmatrix problems generally fall into two categories:
\begin{enumerate}
	\item[\rm(i)] Given a fixed spectrum $\Lambda$, determine all $R \in M_{n}(\mathbb{R})$ for which $R\Lambda$ is a spectrum of $\mu$ (the first type spectral eigenmatrix);
	\item[\rm(ii)] Determine all matrices $R \in M_{n}(\mathbb{R})$ for which there exists a spectrum $\Lambda$ such that $R\Lambda$ is a spectrum of $\mu$ (the second type spectral eigenmatrix).
\end{enumerate}
%In both cases,  $R$ is called a spectral eigenmatrix of $\mu$ and $\Lambda$ is called
%a eigenspectrum of $\mu$ corresponding to $R$.

Spectral eigenmatrix problems are also referred to as scaling matrix problems. The spectral eigenmatrix problem is not only interesting in itself, but also closely related to other relevant fields such as Fourier analysis and number theory \cite{DH, DJ1}. The study of spectral eigenmatrix problems dates back to \L aba and Wang \cite{LWInvent}, who first identified a countable set $\Lambda$ such that both $\Lambda$ and $2\Lambda$ are spectra of a measure $\mu$. In 2012, Dutkay and Jorgensen \cite{DJ1} proved that $R = 5^{k}$ ($k \in \mathbb{N}$) is a spectral eigenmatrix of $\mu_{4,\{0,2\}}$ with respect to the spectrum
\[
\Lambda := \left\{ \sum_{k=1}^{n} 4^{k-1}d_{k} : d_{k} \in \{0, 1\},\ n\in\mathbb{N} \right\}.
\]
Dutkay and Haussermann \cite{DH} showed that if $p$ is a prime greater than $3$, then $p^{n}\Lambda$ is also a spectrum of $\mu_{4,\{0,2\}}$ for any integer $n \ge 1$. However, finding all spectra of an arbitrary singular spectral measure $\mu$ is generally a difficult problem. Recently, Lu obtained a rather complete result; interested readers may consult \cite {L1}. In 2022, An, Dong and He \cite{ADH} studied the spectral eigenmatrix problems of the Sierpinski-type measure $\mu_{Q,D}$ generated by an expanding matrix $Q=\operatorname{diag}[3q,3q]$ and $D=\{(0,0)^{t},(1,0)^{t},(0,1)^{t}\}$. Subsequently, Liu, Tang and Wu \cite{LTW} discussed the spectral eigenmatrix problems of the planar self-affine measure $\mu_{Q,D}$ generated by an expanding integer matrix $Q\in M_{2}(2\mathbb{Z})$ and the four-element digit set $D=\{(0,0)^{t},(1,0)^{t},(0,1)^{t},(-1,-1)^{t}\}$. For more on spectral eigenmatrix problems, see \cite{LA,LLTW,HTW,LW,FHW,WW} and so on.

Let $\mu_{Q,D}$ denote the self-affine measure generated by a finite digit set $D \subset \mathbb{Z}^{n}$ (satisfying \eqref{1.3}) and an expanding matrix
$Q = p\mathrm{diag}[l_{1}, l_{2}, \dots, l_{n}] \in pM_{n}(\mathbb{Z})$,
where $p>2$ is prime.
Motivated by the works mentioned above, the aim of the present paper is to investigate the second type spectral eigenmatrix of  $\mu_{Q,D}$.

Theorem \ref{thm1.2} provides a sufficient condition for a maximal orthogonal set of $\mu_{Q,D}$ to be spectral.
We also use it to establish a sufficient condition for a matrix to be the second type spectral eigenmatrix of $\mu_{Q,D}$.
%Our main result on spectral eigenmatrices is presented in the following theorem.
\begin{thm}\label{thm1.3}
	Let $Q = p\mathrm{diag}[q,\dots, q] \in pM_{n}(\mathbb{Z})$ with prime $p>2$ and integer $q > 1$, $D\subset\mathbb{Z}^{n}$ satisfy \eqref{1.3}.
	Then $R\in M_{n}(\mathbb{Z})$ is a second type spectral eigenmatrix of the spectral measure $\mu_{Q,D}$
	if and only if there exists $k\in\{1,2,\dots,p-1\}$ such that $R\bm{a}\equiv k\bm{a}\pmod {p\mathbb{Z}^{n}} $.
\end{thm}
When \( R \) is a real matrix, determining its spectral eigenmatrix can be challenging. The following theorem addresses this problem specifically for the case when \( R \) is a real diagonal matrix. Under the same assumptions, it also considers a more general form for the expansion matrix \( Q \).
\begin{thm}\label{thm1.4}
	Let $Q = p\mathrm{diag}[l_{1}, l_{2}, \dots, l_{n}] \in pM_{n}(\mathbb{Z})$ with  prime $p>2$ and $|l_{1}l_{2}\dots l_{n}| > 1$, $D\subset\mathbb{Z}^{n}$ satisfy \eqref{1.3}.
	Then a real diagonal matrix $R \in M_{n}(\mathbb{R})$ is a second type  spectral eigenmatrix of the spectral measure $\mu_{Q,D}$ if and only if \( R = \diag[\frac{p_{1}}{q_{1}}, \frac{p_{2}}{q_{2}}, \dots, \frac{p_{n}}{q_{n}}] \in M_{n}(\frac{\mathbb{Z} \backslash p\mathbb{Z}}{\mathbb{Z} \backslash p\mathbb{Z}}) \)
	and there exists $k \in \{1,2,\dots,p-1\}$ such that
	\[
	\mathrm{diag}[p_{1},\dots,p_{n}]\,\bm{a}
	\equiv k\,\mathrm{diag}[q_{1},\dots,q_{n}]\,\bm{a} \pmod{p\mathbb{Z}^{n}}.
	\]
\end{thm}
\begin{remark}
For digit set $D$ satisfying \eqref{1.3}, there exists a vector $\bm{a}$ corresponding to the vector $\bm{a}$ in Theorems \ref{thm1.3} and \ref{thm1.4}. Specifically, in \( \mathbb{R}^{2} \), if the matrix \( Q = \diag[3q, 3q] \) and the digit set \( D = \{(0,0)^{t}, (1,0)^{t}, (0,1)^{t}\} \), then it is straightforward to verify that \( \bm{a} = (1,2)^{t} \).
\end{remark}
The structure of this paper is outlined as follows.
In Section 2, we introduce the necessary notation and basic lemmas related to spectral measures, which are used throughout the paper.
In Section 3, we first characterize the structure of maximal orthogonal sets for $\mu_{Q,D}$ and then give a sufficient condition for such sets to be spectra of $\mu_{Q,D}$.
In Section 4, the necessity part of Theorem \ref{thm1.3} is established in Lemma \ref{lem4.2}, whereas the sufficiency part is derived in Theorem \ref{thm4.10}.
In Section 5, we prove Theorem \ref{thm1.4} and illustrate our main results with illustrative examples.

\section{\bf{Preliminaries}}
In the section, we give some preliminary lemmas and notations. Firstly, it is easy to show that $\Lambda$ is an orthogonal set of $\mu$ if and only if $\widehat{\mu}(\lambda-\lambda ')=0$ for any $\lambda\neq \lambda '\in \Lambda$. In other word, $E_{\Lambda}=\{e^{-2\pi i \langle\lambda, x\rangle}:\lambda\in \Lambda\}$ is an orthogonal family of $ L^2(\mu) $ if and only if
 $ (\Lambda-\Lambda)\backslash\{\bm{0}\}\subset\mathcal{Z}( \widehat{\mu})  $. We say that a discrete set $\Lambda$ is a bi-zero set of $\mu$ if
 $  E_\Lambda = \{ e^{2\pi i \langle \lambda, x \rangle} : \lambda \in \Lambda\} $
 is an orthogonal subset of $L^2(\mu)$.
 Since $D\subset\mathbb{Z}^{n}$, it is not difficult to deduce that $\widehat{\delta}_{Q^{-1}D}$ is a $Q\mathbb{Z}^{n}$-period function. From \eqref{eq1.2}, it follows that
\begin{equation*}
	\mathcal{Z}(\widehat{\mu}_{Q,D})=\bigcup_{k=0}^{\infty}Q^{k}\mathcal{Z}(\widehat{\delta}_{Q^{-1}D})=\bigcup_{k=1}^{\infty}Q^{k}\bigcup_{j=1}^{p-1}\left(\frac{j}{p}\bm{a}+\mathbb{Z}^{n}\right).
\end{equation*}
 Define the function
 \begin{equation*}
 	Q_{\mu,\Lambda}(\xi)=\displaystyle{\sum_{\lambda\in \Lambda}}|\widehat{\mu}(\xi+\lambda)|^2,\quad \xi \in \mathbb{R}^{n}.
 \end{equation*}
 The following theorem is a basic criterion for the spectrality of $\mu$.
    \medskip
 \begin{lemma}\cite{JP1}\label{lem2.2}
    Let $\mu$ be a Borel probability measure with compact support in $\mathbb{R}^{n}$, and let $ \Lambda \subset \mathbb{R}^{n} $ be a countable subset. Then
 \begin{enumerate}
 	\item[\rm (i)] $E_{\Lambda}$ is an orthogonal family of $L^2(\mu)$ if and only if $Q_{\mu,\Lambda}(\xi)\leq1$ for $\xi\in \mathbb{R}^{n}$;
 		\item[\rm (ii)]$E_{\Lambda}$ is an orthogonal basis for $L^2(\mu)$ if and only if $Q_{\mu,\Lambda}(\xi)=1$ for $\xi\in \mathbb{R}^{n}$;
 			\item[\rm(iii)]$Q_{\mu,\Lambda}(\xi)$ has an entire analytic extension to $\mathbb{C}^{n}$ if $\Lambda$ is an orthogonal set of $\mu$.
\end{enumerate}
\end{lemma}

Next, we will introduce the definition of a compatible pair.

\begin{defn}\cite{JP1}\label{de2.3}
 Let $Q\in M_{n}(\mathbb{Z})$  be an $n\times n$ expansive matrix with integer entries. Let $D,L \subset \mathbb{Z}^{n}$ be a finite set of integer vectors with $\#D = \#L=N$ and $\bm{0} \in D\cap L$. We say that $(Q^{-1}D, L)$ forms a compatible pair (or $(Q, D)$ is admissible, or $(Q,D,L)$ forms a Hadamard triple)
 if the matrix
 \begin{align*}
 H:=\frac{1}{\sqrt {N}}\left[e^{2\pi i  \langle Q^{-1}d, l  \rangle}\right]_{d\in D, l\in L}
 \end{align*}
 is unitary, i.e., $HH^*=I$, where $H^*$ denotes the transposed conjugate of $H$.
\end{defn}

It is very convenient to construct the orthogonal exponential function family of $ L^{2}(\mu) $ using Hadamard triple. However, the challenge lies in verifying that the constructed $ \Lambda $ forms an orthonormal basis for $ L^{2}(\mu) $. Nowadays, more and more results indicate the mysterious relationship between Hadamard triples and spectral measures \cite{LW1,S,S1}.

\begin{lemma}\label{lem2.3}
Let $Q\in M_{n}(\mathbb{Z})$ be an expanding matrix, and let $D,S\subset\mathbb{Z}^{n}$ be two finite digit sets with $\#D=\#S.$ Then
\begin{enumerate}
	\item[\rm (i)] $(Q,D,S)$ is a Hadamard triple if and only if $\widehat{\delta}_{Q^{-1}D}(s_{1}-s_{2})=0$ for any $s_{1}\neq s_{2}\in S$;
\item[\rm (ii)] Suppose that all $ (Q_{k}^{-1}D_{k},S_{k}) $ are integral compatible pairs for $ k\ge1 $. Then
\[ ((Q_{k}Q_{k-1}\dots Q_{1})^{-1}\widetilde{D}_{k},\widetilde{S}_{k})  \]
is an integral compatible pair for each $ k\ge1 $, where    \[ \widetilde{S}_{k}=S_{1}+Q_{1}^{t}S_{2}+\dots+Q_{1}^{t}Q_{2}^{t}\dots Q_{k-1}^{t}S_{k},\quad\widetilde{D}_{k}=D_{k}+Q_{k}D_{k-1}+\dots+Q_{k}Q_{k-1}\dots Q_{2}D_{1 } .  \]
\end{enumerate}
\end{lemma}

\begin{lemma}\cite{S1}\label{lem2.7}
	Let $\mu$ be defined by \eqref{1.1} and let $\{(Q,D,S_{k})\}_{k=1}^{\infty}$ be a Hadamard triple sequence with $\bm{0}\in S_{k}.$ If $\mathcal{Z}(\widehat{\mu})$ is uniformly disjoint from the sets $Q^{-k}S_{1}+Q^{-(k-1)}S_{2}+\dots+Q^{-1}S_{k}$ for all large $k$, then $\Lambda=S_{1}+QS_{2}+Q^{2}S_{3}+\dots$ is a spectrum of $\mu.$
\end{lemma}

From now until before Section 5, we assume \( Q =p\diag[q,q,\dots,q]\) with prime $p>2$ and integer $q\geq1$. The following concept is useful in considering spectral eigenmatrix problems.

\begin{defn}
	Let $R \in M_{n}(\mathbb{Z})$ be an expanding matrix and let $D$ be
	a digit set in $\mathbb{Z}^{n}$. The finite set $S = \{x_{0},\dots,x_{s-1}\}$ is called a cycle with
	length $s$ in $T(R, D)$ if there exists $\{b_{0},\dots,b_{s-1}\}\subset D$ such that
\begin{align*}
	x_{1}=R^{-1}(x_{0}+b_{0}),\dots,\;x_{s-1}=R^{-1}(x_{s-2}+b_{s-2}),\;x_{0}=R^{-1}(x_{s-1}+b_{s-1}):=x_{s}.
\end{align*}
In addition, if $|\widehat{\delta}_{D}(x_{i})|^{2}=1$ for all $i$, then $S$ is called a $D-cycle$.
\end{defn}
\begin{lemma}\cite{DJ}\label{lem2.4}
	Suppose that $(Q, D, L)$ forms a Hadamard triple with $\bm{0} \in D \cap L$. Then $\Lambda(Q, L)$ is a spectrum of $\mu$ if and only if $S = \{\bm{0}\}$ is the unique $D$-cycle in $T(Q, D)$, where
	\begin{equation*}
		\Lambda(Q, L) = \left\{\sum_{j=1}^{k} Q^{j-1}L : k \in \mathbb{N}\right\}.
	\end{equation*}
\end{lemma}
 Similar to \cite[Lemma 5.7]{ADH}, we also have the following lemma.
\begin{lemma}\label{lem2.5}
Let $x \in T(Q, \Gamma_{q})$. Then $x$ can be expressed by
\begin{equation}\label{eq2.1}
	x=\sum_{k=1}^{\infty}Q^{-k}x_{k}=\sum_{k=1}^{\infty}(pq)^{-k}x_{k},
\end{equation}
where $x_{k}\in\Gamma_{q}$. Moreover, if $x\in\mathbb{Q}^{n}$, then the expression of \eqref{eq2.1} is eventually.
\end{lemma}
\begin{pro}\label{pro2.9}
	Let $A\subset\Gamma_{q}$, then the following statements are equivalent:
	\begin{enumerate}
		\item[\rm(i)] $A$ is a maximal orthogonal set of $\mu_{Q,D}$ in $\Gamma_{q}$.
		\item[\rm(ii)] $A$ is a spectrum of $\delta_{Q^{-1}D}$.
		\item[\rm(iii)] For any $a\in A$, $A\equiv a+A_{q}\pmod Q$.
	\end{enumerate}
\end{pro}

\begin{proof}
	$\rm(i)\Leftrightarrow \rm(ii)$ We first prove that $A\subset\Gamma_{q}$ is an orthogonal set of $\mu_{Q,D}$ if and only if $A$ is an orthogonal set of $\delta_{Q^{-1}D}$. In fact $\Gamma_{q}\subset [-\frac{pq}{2},\frac{pq}{2})^{n}$, we know that
	\begin{equation*}
		(A-A)\backslash\{\bm{0}\}\subset (-pq,pq)^{n}\cap\mathcal{Z}(\widehat{\mu}_{Q,D})\subset\mathcal{Z}(\widehat{\delta}_{Q^{-1}D}).
	\end{equation*}
	Therefore, the necessity has been established.  The sufficiency follows from $\mathcal{Z}(\widehat{\delta}_{Q^{-1}D})\subset\mathcal{Z}(\widehat{\mu}_{Q,D}) .$
	
	Suppose $A$ is a maximal orthogonal set of $\mu_{Q,D}$ in $\Gamma_{q}$, then it is also an orthogonal set of $\delta_{Q^{-1}D}$. If $\#A < p$, without loss of generality, we suppose $\#A = p-1$ and let $A = \{a_1, a_2, \dots, a_{p-1}\}$ such that
	\begin{equation*}
		a_i - a_1 \equiv iq\bm{a} \pmod{Q} \in \mathcal{Z}(\widehat{\delta}_{Q^{-1}D}), \quad i = 2, 3, \dots, p-1.
	\end{equation*}
Denote $a \equiv a_1 + q\bm{a} \pmod{Q}$ with $a \in \Gamma_q$. Then $a - a_1 \equiv q\bm{a} \pmod{Q}$, and
\begin{equation*}
	a - a_i = (a - a_1) - (a_i - a_1) \equiv (1 - i)q\bm{a} \pmod{Q}.
\end{equation*}
Thus, for any $j \in \{1, \dots, p-1\}$, $a - a_j \in \mathcal{Z}(\widehat{\delta}_{Q^{-1}D})$, which implies that $A \cup \{a\}$ is also an orthogonal set with respect to $\mu_{Q,D}$. This contradicts the maximality of $A$. Hence, we have $\#A = p$, which is the dimension of $L^2(\delta_{Q^{-1}D})$. In summary, $A$ is the spectrum of $\delta_{Q^{-1}D}$.

	On the other hand, if $A$ is a spectrum of $\delta_{Q^{-1}D}$, then for any $a \in \Gamma_{q}\backslash A,\; A\cup\{a\}$ is not an orthogonal set of $\mu_{Q,D}$. Therefore, $A$ is a maximal orthogonal set of $\mu_{Q,D}$ in $\Gamma_{q}$.

$\rm(ii)\Rightarrow \rm(iii)$ Let $A$ be a spectrum of $\delta_{Q^{-1}D}$. We then know $\#A = p$ and take $A = \{a_0, a_1, a_2, \dots, a_{p-1}\}$ for convenience. For any $a \in A$, we take $a = a_0$ for convenience and can rewrite $A = \{a_0, b_1+a_0, b_2+a_0, \dots, b_{p-1}+a_0\}$, where $b_i = a_i - a_0$ for $i = 1, 2, \dots, p-1$. By orthogonality, we conclude that the sets satisfy $\{b_1, b_2, \dots, b_{p-1}\} \equiv \{q\bm{a}, 2q\bm{a}, \dots, (p-1)q\bm{a}\} \pmod{Q}$. It follows that $A \equiv a_0 + A_q \pmod{Q}$.

$\rm(iii)\Rightarrow \rm(ii)$
For any $a\in A$, $A\equiv a+A_{q}\pmod{Q}$. This implies that $\#A=p$ and
\[
A-A \equiv A_{q}-A_{q} \pmod{Q}.
\]
Since $A_q$ is a spectrum of $\delta_{Q^{-1}D}$, it readily follows that $A \equiv a + A_q \pmod{Q}$ is also a spectrum of $\delta_{Q^{-1}D}$.
\end{proof}

\begin{pro}\label{pro2.10}
Let $A$ be a maximal orthogonal set of $\mu_{Q,D}$ in $\Gamma_{q}$.
Then there exists a unique
$a_{1}:=(a_{1,1},\dots,a_{n,1})^{t}\in A$ satisfying
$-\frac{q}{2}\le a_{1,1}<\frac{q}{2}$.
Moreover, the map from $A$ to $a_{1}$ is one-to-one.
\end{pro}
\begin{proof}
	Let $\Psi_q = [-\frac{pq}{2}, \frac{pq}{2})$ be a circle group, and define $d(x,y) = |x-y| \pmod{pq}$ for any $x, y \in \Psi_q$. By Proposition \ref{pro2.9}, $A$ is a spectrum of $\delta_{Q^{-1}D}$. Hence, $\#A = p$, and we write $A = \{a_1, a_2, \dots, a_p\}$, where $a_i = (a_{1,i}, a_{2,i}, \dots, a_{n,i})^t$ for $i \in \{1, 2, \dots, p\}$. The orthogonal property of $A$ implies that
	\begin{equation*}
		(A-A)\setminus\{\bm{0}\}\subseteq A_q \pmod{Q}.
	\end{equation*}
Then, for any two distinct elements $a_{1,j}$ and $a_{1,k}$ in $\{a_{1,1},a_{1,2},\dots,a_{1,p}\}$, we have $d(a_{1,j}, a_{1,k}) \geq q$. Furthermore, each of $a_{1,1},a_{1,2},\dots,a_{1,p}$ lies in a distinct interval:
\begin{equation*}
	\underbrace{
		\left[-\frac{pq}{2}, -\frac{(p-2)q}{2}\right),\ \left[-\frac{(p-2)q}{2}, -\frac{(p-4)q}{2}\right),\dots,\ \left[-\frac{q}{2}, \frac{q}{2}\right),\dots, \left[\frac{(p-2)q}{2}, \frac{pq}{2}\right)
	}_{p}
\end{equation*}
The first statement follows.
	
	Let $A \neq A'$ be two maximal orthogonal sets of $\mu_{Q,D}$ in $\Gamma_q$. For the second conclusion, we only need to prove that $A$ and $A'$ are disjoint. For contradiction, suppose there exists an element $a$ such that $a \in A \cap A'$. By Proposition \ref{pro2.9}, we have
	\begin{equation*}
		A \equiv A' \equiv a + A_q \pmod{Q}.
	\end{equation*}
	Since $A, A' \subset \Gamma_q$, this implies $A = A'$, a contradiction.
\end{proof}
 Based on the above two lemmas, the following decomposition of $\Gamma_q$ holds.
\begin{pro}\label{pro2.11}
Let $B_q = \{(i_1, i_2, \dots, i_n)^t \in \Gamma_q : -\tfrac{q}{2} \le i_1 < \tfrac{q}{2}\}$ and $A_q = q\{\bm{0}, \bm{a}, \dots, (p-1)\bm{a}\}$. We decompose $\Gamma_q$ into a disjoint union:
\begin{equation*}
	\Gamma_q = \bigcup_{\bm{b} \in B_q} \mathcal{A}_{\bm{b}},
\end{equation*}
where $\mathcal{A}_{\bm{b}} \equiv \bm{b} + A_q \pmod{Q}$ is a maximal orthogonal set for $\mu_{Q,D}$ in $\Gamma_q$.
\end{pro}
\begin{proof}
For any $\bm{b} \in B_q$, since $\mathcal{A}_{\bm{b}} \subset \Gamma_q$, we have $\cup_{\bm{b} \in B_q} \mathcal{A}_{\bm{b}} \subset \Gamma_q$. On the other hand, take any $\lambda \in \Gamma_q$ and define $\mathcal{A}_{\lambda} \equiv \lambda + A_q \pmod{Q}$. Then $\lambda \in \mathcal{A}_{\lambda}$, and Proposition \ref{pro2.9} implies that $\mathcal{A}_{\lambda}$ is a maximal orthogonal set for $\mu_{Q,D}$ in $\Gamma_q$. According to Proposition \ref{pro2.10}, there exists a unique element $\bm{b} \in \mathcal{A}_{\lambda} \cap B_q$. Furthermore, by Proposition \ref{pro2.9}, we have $\mathcal{A}_{\lambda} \equiv \bm{b} + A_q \pmod{Q}$. Thus, $\lambda \in \mathcal{A}_{\bm{b}}$, and hence $\Gamma_q \subset \cup_{\bm{b} \in B_q} \mathcal{A}_{\bm{b}}$. The disjointness of the union follows from the proof of Proposition \ref{pro2.10}.
\end{proof}
Next, we present an example to illustrate the decomposition of $\Gamma_q$.
\begin{exam}\label{exam2.11}
	Let $n=2$, $p=3$, $q=2$ and $\bm{a}=(1,1)^t$. Then
	\[
	\Gamma_q = \left\{(i,j)^t \in \mathbb{Z}^2 : -3 \le i,j < 3\right\},\;\;
	B_q = \left\{(i,j)^t \in \Gamma_q : -1 \le i < 1\right\},
	\]
	and
	\[
	A_q = 2\{\bm{0}, \bm{a}, 2\bm{a}\}=2\left\{(0,0)^t, (1,1)^t, (2,2)^t\right\}.
	\]
By direct computation, we have
\begin{align*}
	&\mathcal{A}_{(-1,-3)^t} = \left\{(-1,-3)^t, (1,-1)^t, (-3,1)^t\right\}, &
	&\mathcal{A}_{(-1,-2)^t} = \left\{(-1,-2)^t, (1,0)^t, (-3,2)^t\right\},\\
	&\mathcal{A}_{(-1,-1)^t} = \left\{(-1,-1)^t, (1,1)^t, (-3,-3)^t\right\}, &
	&\mathcal{A}_{(-1,0)^t} = \left\{(-1,0)^t, (1,2)^t, (-3,-2)^t\right\},\\
	&\mathcal{A}_{(-1,1)^t} = \left\{(-1,1)^t, (1,-3)^t, (-3,-1)^t\right\},  &
	&\mathcal{A}_{(-1,2)^t} = \left\{(-1,2)^t, (1,-2)^t, (-3,0)^t\right\},\\
	&\mathcal{A}_{(0,-3)^t} = \left\{(0,-3)^t, (2,-1)^t, (-2,1)^t\right\}, &
	&\mathcal{A}_{(0,-2)^t} = \left\{(0,-2)^t, (2,0)^t, (-2,2)^t\right\},\\
	&\mathcal{A}_{(0,-1)^t} = \left\{(0,-1)^t, (2,1)^t, (-2,-3)^t\right\},  &
	&\mathcal{A}_{(0,0)^t} = \left\{(0,0)^t, (2,2)^t, (-2,-2)^t\right\},\\
	&\mathcal{A}_{(0,1)^t} = \left\{(0,1)^t, (2,-3)^t, (-2,-1)^t\right\}, &
	&\mathcal{A}_{(0,2)^t} = \left\{(0,2)^t, (2,-2)^t, (-2,0)^t\right\}.
\end{align*}
It follows that
\begin{equation*}
	\Gamma_q = \bigcup_{(i,j)^t \in B_q} \mathcal{A}_{(i,j)^t}.
\end{equation*}
For any $(i,j)^t \in B_q$, the set $\mathcal{A}_{(i,j)^t} \equiv (i,j)^t + A_q \pmod{6}$ is a maximal orthogonal set for $\mu_{Q,D}$ in $\Gamma_q$.
\end{exam}

\section{ \bf{Maximal orthogonal sets and spectra} }
In this section, we focus on proving results concerning the spectral structure. To begin, we prove Theorem \ref{thm1.1}, which states that a maximal orthogonal set of $\mu_{Q,D}$ can be expressed by the maximal mapping $\gamma$. Below is an intuitive illustration of the maximal map $\gamma$.
\par % 换行（可选，若要图片在语句下一行）
\noindent % 取消首行缩进
\begin{minipage}{\textwidth}
	\centering
	% 第一张图
	\begin{minipage}{0.82\textwidth}
		\centering
		\includegraphics[angle=90, width=\textwidth]{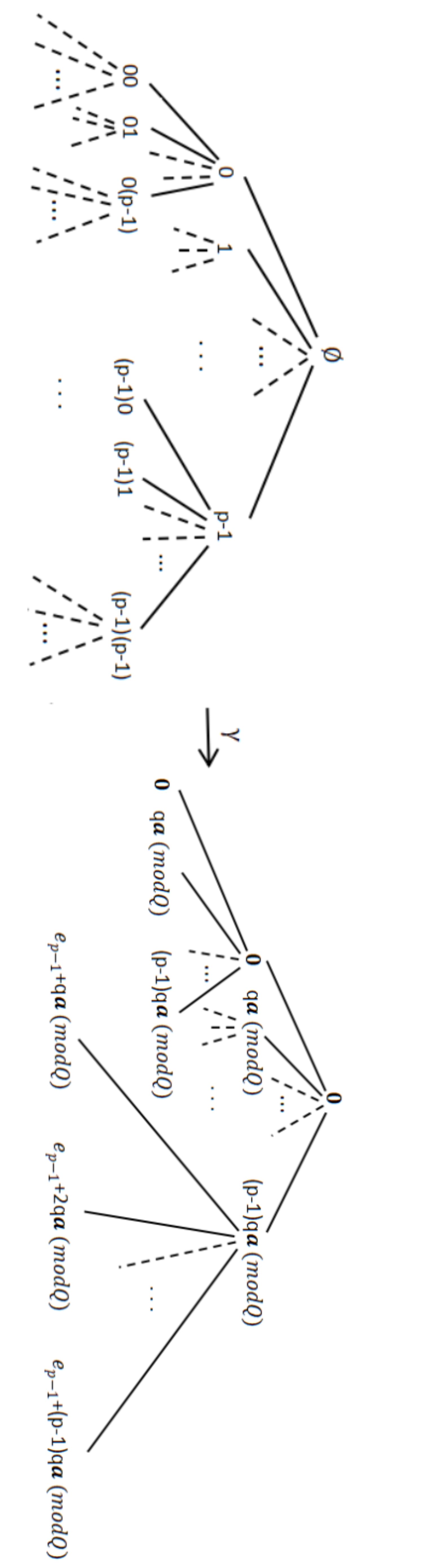}
		\footnotesize Figure 1. Illustration of a maximal mapping $\gamma$
		\label{tree1}
	\end{minipage}
\end{minipage}

\begin{proof}[\textbf{Proof of Theorem \ref{thm1.1}}]
We first prove the sufficiency. For any $\mathbf{I} \neq \mathbf{J} \in \Sigma_p^\gamma$, let $k$ be the smallest positive integer such that $\mathbf{I}|_k \neq \mathbf{J}|_k$. It is easy to see that
\begin{align*}
	\gamma^*(\mathbf{I}) - \gamma^*(\mathbf{J}) &= (pq)^{k-1}(\gamma(\mathbf{I}|_k) - \gamma(\mathbf{J}|_k) + pqz_1) \\
	&= (pq)^{k-1}((i_k - j_k)q\bm{a} + pqz_2) \in \mathcal{Z}(\widehat{\delta}_{Q^{-k}D}) \subset \mathcal{Z}(\widehat{\mu}_{Q,D}),
\end{align*}
where $z_1, z_2 \in \mathbb{Z}^n$. Thus, by the properties of the bi-zero set, it follows that $\gamma^*(\Sigma_p^\gamma)$ forms an orthogonal set for $\mu_{Q,D}$. Next, we demonstrate that $\gamma^*(\Sigma_p^\gamma)$ is a maximal orthogonal set for $\mu_{Q,D}$. Suppose there exists $l \in \mathbb{R}^n \setminus \gamma^*(\Sigma_p^\gamma)$ such that $\{l\} \cup \gamma^*(\Sigma_p^\gamma)$ is also an orthogonal set of $\mu_{Q,D}$. Since $\bm{0} \in \gamma^*(\Sigma_p^\gamma)$ and $\mathcal{Z}(\widehat{\mu}_{Q,D}) \subset \mathbb{Z}^n$, it follows that $l \in \mathcal{Z}(\widehat{\mu}_{Q,D}) \subset \mathbb{Z}^n$. Thus, via the $Q$-adic expansion, $l$ admits the representation
\begin{equation*}
	l = \sum_{k=1}^{\infty} c_k (pq)^{k-1} \quad \text{with } c_k \in \Gamma_q,
\end{equation*}
where $c_k = \bm{0}$ for all sufficiently large $k$. We assert that there exists $i_1 \in \Sigma_p$ such that $\gamma(i_1) = c_1$.
Suppose to the contrary that this assertion fails; then for any $i \in \Sigma_p$, there exists $\mathbf{J}_i$ for which
\[
\gamma^*(i\mathbf{J}_i0^\infty) - l = \gamma(i) - c_1 + pq\omega \in \mathcal{Z}(\widehat{\mu}_{Q,D}),
\]
where $\omega \in \mathbb{Z}^n$. Since $\gamma(i) - c_1 \in (-pq, pq)^n \setminus\{\bm{0}\}$, we have
\begin{equation*}
	\gamma(i) - c_1 \in \mathcal{Z}(\widehat{\delta}_{Q^{-1}D}) + pq\mathbb{Z}^n.
\end{equation*}
This implies that $\{\gamma(i) : i \in \Sigma_p\} \cup \{c_1\}$ is an orthogonal set with respect to $\delta_{Q^{-1}D}$. This, however, contradicts the fact that $\{\gamma(i) : i \in \Sigma_p\}$ is a spectrum of $\delta_{Q^{-1}D}$, and thus our assertion is valid. Therefore, there exists $i_1 \in \Sigma_p$ such that $\gamma(i_1) = c_1$. Next, by the same reasoning applied to $c_2$, we deduce that there exists $i_2 \in \Sigma_p$ such that $\gamma(i_1i_2) = c_2$. Collectively, we infer that there exists $\mathbf{I} \in \Sigma_p^\gamma$ for which $\gamma^*(\mathbf{I}) = l$. This establishes that $\Lambda = \gamma^*(\Sigma_p^\gamma)$ is a maximal orthogonal set.
	
For the necessity, suppose that $\Lambda$ is a maximal orthogonal set of $\mu_{Q,D}$, then $\Lambda \subset \mathcal{Z}(\widehat{\mu}_{Q,D}) \subset \mathbb{Z}^n$. Let $\Lambda = \{\lambda_n\}_{n=0}^\infty$ with $\lambda_0 = \bm{0}$, where each $\lambda_n$ has the $Q$-adic expansion
\[
\lambda_n = \sum_{k=1}^\infty c_{n,k} (pq)^{k-1}
\]
with $c_{n,k} \in \Gamma_q$ and $c_{n,k} = \bm{0}$ for all sufficiently large $k$. Denote $C_{\emptyset}=\{c_{n,1}:n\ge0\}$, where $c_{0,1}=\bm{0}. $ For any distinct $l_1, l_2 \in C_\emptyset$, there exist indices $n, m$ such that $l_1 = c_{m,1}$, $l_2 = c_{n,1}$, and
\begin{equation*}
	l_1 - l_2 = c_{m,1} - c_{n,1} \equiv \lambda_m - \lambda_n \pmod{Q}, \quad l_1 - l_2 \in \mathcal{Z}(\widehat{\delta}_{Q^{-1}D}).
\end{equation*}
It follows that $C_\emptyset$ is an orthogonal set for $\delta_{Q^{-1}D}$. We now show that $C_\emptyset$ is the maximal orthogonal set of $\delta_{Q^{-1}D}$ within $\Gamma_q$. Suppose otherwise; then there exists $l_3 \in \Gamma_q$ such that $C_\emptyset - l_3 \subset \mathcal{Z}(\widehat{\delta}_{Q^{-1}D})$.  Furthermore, $\Lambda - l_3 \subset \mathcal{Z}(\widehat{\delta}_{Q^{-1}D}) + pq\mathbb{Z}^n \subset \mathcal{Z}(\widehat{\mu}_{Q,D})$, which contradicts the maximality of $\Lambda$. We thus conclude that $C_\emptyset$ is the maximal orthogonal set of $\delta_{Q^{-1}D}$ in $\Gamma_q$. We define a mapping $\gamma: \Sigma_p^* \to \Gamma_q$ by $\gamma(i) = iq\bm{a} \pmod{Q} \in \Gamma_q$ for all $i \in \Sigma_p$, and we set
 \begin{equation*}
 	C_i = \{ c_{n,2} : c_{n,1} = \gamma(i),\ n \ge 0 \}.
 \end{equation*}
For any $i \in \Sigma_p$, it is easy to verify that $C_i$ is a nonempty set. Similarly, we can show that $C_i$ is also a maximal orthogonal set of $\delta_{Q^{-1}D}$ within $\Gamma_q$ for any $i \in \Sigma_p$. By Proposition \ref{pro2.9}, $C_i \equiv e_i + A_q \pmod{Q}$ for some $e_i \in B_q$. We next extend the definition to $\gamma(ij) = e_i + jq\bm{a} \pmod{Q} \in \Gamma_q$. For any $\mathbf{I} \in \Sigma_p^k$, the set
\begin{equation}\label{eq2.2}
	C_{\mathbf{I}} = \left\{ c_{n,k+1} : c_{n,1} = \gamma(\mathbf{I}|_1),\ c_{n,2} = \gamma(\mathbf{I}|_2),\dots,\ c_{n,k} = \gamma(\mathbf{I}|_k),\ n \ge 0 \right\} \subset \Gamma_q
\end{equation}
is nonempty. For any distinct $r_1, r_2 \in C_{\mathbf{I}}$, there exist $n, m$ such that $\lambda_m, \lambda_n \in \Lambda$ with $r_1 = c_{m,k+1}$ and $r_2 = c_{n,k+1}$ (where $r_1 \neq r_2$). We then have
\begin{equation*}
	\lambda_m - \lambda_n = (pq)^k (r_1 - r_2 + pqz_1) \in \mathcal{Z}(\widehat{\mu}_{Q,D}).
\end{equation*}
Since $r_1, r_2 \in \Gamma_q$, it follows that $r_1 - r_2 \in \mathcal{Z}(\widehat{\delta}_{Q^{-1}D})$. Similarly, we can show that $C_{\mathbf{I}}$ is also a maximal orthogonal set of $\delta_{Q^{-1}D}$ within $\Gamma_q$; the details are omitted here for brevity. By Proposition \ref{pro2.9} again, $C_{\mathbf{I}} \equiv e_{\mathbf{I}} + A_q \pmod{Q}$ for some $e_{\mathbf{I}} \in B_q$. We then extend $\gamma$ to act on $\mathbf{I}j$ via $\gamma(\mathbf{I}j) = e_{\mathbf{I}} + jq\bm{a} \pmod{Q} \in \Gamma_q$. By induction, the mapping $\gamma$ is well-defined. We now establish that $\Lambda = \gamma^*(\Sigma_p^\gamma)$.  For any $n > 0$, one has
	\begin{equation*}
		\lambda_{n}=\sum_{k=1}^{\infty}c_{n,k}(pq)^{k-1}	
	\end{equation*}
	with $c_{n,N}\neq\bm{0}$ and $c_{n,k}=\bm{0} $ for $k>N$. Based on the above discussion, we may select $\mathbf{I} \in \Sigma_p^N$ for which \eqref{eq2.2} holds. Note that $\bm{0} \in C_{\mathbf{I}}$ due to $c_{n,N+1} = \bm{0}$; thus, $C_{\mathbf{I}} \equiv A_q \pmod{Q}$. In general, we have $C_{\mathbf{I}0^l} \equiv A_q \pmod{Q}$ for all $l \ge 1$, as $c_{n,k} = \bm{0}$ for all $k > N$. This implies that $\lambda_n = \gamma^*(\mathbf{I})$ by definition of $\gamma$. Hence, $\Lambda\subset\gamma^{*}(\Sigma_{p}^{\gamma})$. Conversely, for each $\mathbf{I} \in \Sigma_p^\gamma$, $\gamma^*(\mathbf{I}) = \sum_{k=1}^\infty \gamma(\mathbf{I}|_k) (pq)^{k-1}$. Let $N$ be a positive integer with $\gamma(\mathbf{I}|_N) \neq \bm{0}$ and $\gamma(\mathbf{I}|_k) = \bm{0}$ for all $k > N$. By the definition of $\gamma$ following \eqref{eq2.2}, we have $C_{\mathbf{I}|_k} \equiv A_q \pmod{Q}$ for all $k > N$. Then, the set
	\begin{equation*}
		\left\{ n : c_{n,1} = \gamma(\mathbf{I}|_1),\ c_{n,2} = \gamma(\mathbf{I}|_2),\ \dots,\ c_{n,N} = \gamma(\mathbf{I}|_N),\ c_{n,k} = \gamma(\mathbf{I}|_k) = \bm{0} \text{ for all } k \ge N \right\}
	\end{equation*}
	is nonempty. This implies that there exists an $n$ such that $\lambda_n = \gamma^*(\mathbf{I})$, and thus $\Lambda = \gamma^*(\Sigma_p^\gamma)$.
	
	Lastly, we establish that $\gamma$ is a maximal mapping, i.e., it satisfies Definition \ref{defn1.3}. It is clear that $\gamma$ satisfies conditions (i) and (ii) in the definition of maximal mappings. To verify Condition (iii) of Definition \ref{defn1.3} for $\gamma$, fix an arbitrary $\mathbf{I} \in \Sigma_p^*$. By \eqref{eq2.2}, $C_{\mathbf{I}}$ is nonempty. Since $\Lambda = \gamma^*(\Sigma_p^\gamma)$, there exists $\lambda \in \Lambda$ and $\mathbf{J} \in \Sigma_p^*$ such that $\mathbf{IJ} \in \Sigma_p^\gamma$ and $\lambda = \gamma^*(\mathbf{IJ})$. For all sufficiently large $n$, we have $\gamma(\mathbf{IJ}0^n) = \bm{0}$, from which the assertion follows.
\end{proof}
	From Theorem \ref{thm1.1}, we already know that $\Lambda = \gamma^*(\Sigma_p^\gamma)$ is a maximal orthogonal set of $\mu_{Q,D}$ if and only if $\gamma$ is a maximal tree mapping. In the following, a simple conclusion can be proved.

\begin{lemma}\label{lem3.1}
Let $\gamma$ be a maximal mapping and let $k \ge 1$ be a positive integer. Then, $\Lambda = \gamma^*(\Sigma_p^k)$ is the spectrum of $\mu_k$.
\end{lemma}
\begin{proof}
	For any $\mathbf{I}, \mathbf{J} \in \Sigma_p^k$ with $\mathbf{I} \neq \mathbf{J}$, let $l$ be the smallest positive integer such that $\mathbf{I}|_l \neq \mathbf{J}|_l$. Then,
	\begin{equation*}
		\gamma^*(\mathbf{I}) - \gamma^*(\mathbf{J}) = \sum_{j=l}^k (pq)^{j-1} ( \gamma(\mathbf{I}|_j) - \gamma(\mathbf{J}|_j) ) \in \mathcal{Z}(\widehat{\delta}_{Q^{-l}D}) \subset \mathcal{Z}(\widehat{\mu}_k).
	\end{equation*}
Thus, $\{\gamma^*(\mathbf{I}) : \mathbf{I} \in \Sigma_p^k\}$ is an orthogonal set for $\mu_k$. Moreover, $\#\{\gamma^*(\mathbf{I}) : \mathbf{I} \in \Sigma_p^k\} = p^k$, which coincides with the dimension of $L^2(\mu_k)$. We therefore conclude that $\{\gamma^*(\mathbf{I}) : \mathbf{I} \in \Sigma_p^k\}$ is the spectrum of $\mu_k$.
\end{proof}

\begin{lemma}\label{lem3.2}
	Let $\gamma$ be a maximal mapping from $\Sigma_{p}^{*}$ to $\Gamma_{q}$. If for each $\mathbf{I} \in \Sigma_{p}^{*}$ there exists a $\mathbf{J}_{\mathbf{I}}$ such that $\mathbf{I}\mathbf{J}_{\mathbf{I}} \in \Sigma_{p}^{\gamma}$ and $\sup_{\mathbf{I} \in \Sigma_{p}^{*}} \Theta_{\mathbf{I}}(\mathbf{J}_{\mathbf{I}}) < \infty$, then there exists a constant $c$ and a word $\mathbf{L}$ for $\mathbf{I} \in \Sigma_{p}^{k}$ such that $\mathbf{I}\mathbf{L} \in \Sigma_{p}^{\gamma}$ and
	\begin{equation*}
		| \widehat{\mu}_{k}( \xi + \gamma^{*}(\mathbf{I}\mathbf{L}) ) | \le c | \widehat{\mu}_{Q,D}( \xi + \gamma^{*}(\mathbf{I}\mathbf{L}) ) |, \quad k \ge 1.
	\end{equation*}
\end{lemma}
\begin{proof}
For any $\mathbf{I} \in \Sigma_{p}^{k}$, if $e_{\mathbf{I}0^{l}} = \bm{0}$ for $l \ge 0$, then we can choose $\mathbf{J} = 0^{\infty}$ so that $\mathbf{IJ} \in \Sigma_{p}^{\gamma}$ and
	\begin{align}\label{eq3.1}
	\nonumber	\widehat{\mu}_{Q,D}( \xi + \gamma^{*}(\mathbf{IJ}) )
		&= \widehat{\mu}_{k}( \xi + \gamma^{*}(\mathbf{IJ}) )
		\widehat{\mu}_{Q,D}\left( \frac{\xi + \gamma^{*}(\mathbf{IJ})}{(pq)^{k}} \right)\\
		&= \widehat{\mu}_{k}( \xi + \gamma^{*}(\mathbf{I}) )
		\widehat{\mu}_{Q,D}\left( \frac{\xi + \gamma^{*}(\mathbf{I})}{(pq)^{k}} \right).
	\end{align}
Otherwise, let $l$ be the smallest integer such that $e_{\mathbf{I}0^{l}} \neq \bm{0}$. Without loss of generality, we may assume that $l = 1$. Under certain conditions, there exists a $\mathbf{J}_{\mathbf{I}}$ such that $\mathbf{I}\mathbf{J}_{\mathbf{I}} \in \Sigma_{p}^{\gamma}$ and $\sup_{\mathbf{I} \in \Sigma_{p}^{*}} \Theta_{\mathbf{I}}(\mathbf{J}_{\mathbf{I}}) < \infty$. Taking $\mathbf{J} = 0\mathbf{J}_{\mathbf{I}}$, and similar to \eqref{eq2.4}, we obtain the following decomposition:
\begin{equation*}
	\Upsilon_{\gamma}(\mathbf{IJ}) = \Upsilon_{\gamma}(\mathbf{I}) \Upsilon_{\gamma}(\mathbf{IJ})_{k_{0},k_{1}} \dots \Upsilon_{\gamma}(\mathbf{IJ})_{k_{m},\infty},
\end{equation*}
where $k_{0} = k + 1$ and $\gamma(\mathbf{IJ}_{k_{i}}) \neq \bm{0}$. Then,
\begin{align*}
	\gamma^{*}(\mathbf{IJ})&=\gamma^{*}(\mathbf{I})+\sum_{i=0}^{m}\sum_{l=k_{i}}^{k_{i+1}-1}\gamma(\mathbf{IJ}|_{l})(pq)^{l-1}\\
	&=\gamma^{*}(\mathbf{I})+\sum_{i=0}^{m}\sum_{l=k_{i}}^{k_{i}+\ell(\Upsilon_{\gamma}(\mathbf{IJ})_{k_{i},k_{i+1}})-1}\gamma(\mathbf{IJ}|_{l})(pq)^{l-1}.
\end{align*}
In fact, it suffices to consider the case $m = 1$, as the method extends to all other cases identically. When $m = 1$, the above equation reduces to
\begin{equation*}
	\gamma^{*}(\mathbf{IJ}) = \gamma^{*}(\mathbf{I}) + \sum_{l=k_{0}}^{k_{0}+\ell(\Upsilon_{\gamma}(\mathbf{IJ})_{k_{0},k_{1}})-1}\gamma(\mathbf{IJ}|_{l})(pq)^{l-1} + \sum_{l=k_{1}}^{k_{1}+\ell(\Upsilon_{\gamma}(\mathbf{IJ})_{k_{1},k_{2}})-1}\gamma(\mathbf{IJ}|_{l})(pq)^{l-1}.
\end{equation*}
Thus,
\begin{align}\label{eq3.2}
	\begin{split}
		\widehat{\mu}_{Q,D}(\xi+\gamma^{*}(\mathbf{I}\mathbf{J}))&=\widehat{\mu}_{k}(\xi+\gamma^{*}(\mathbf{I}\mathbf{J}))\widehat{\mu}_{Q,D}\left(\frac{\xi+\gamma^{*}(\mathbf{I}\mathbf{J})}{(pq)^{k}}\right)\\
		&=\widehat{\mu}_{k}(\xi+\gamma^{*}(\mathbf{I}))\widehat{\mu}_{Q,D}\left(\frac{\xi+\gamma^{*}(\mathbf{I})}{(pq)^{k}}+\sum_{i=1}^{\ell(\Upsilon_{\gamma}(\mathbf{IJ})_{k_{0},k_{1}})}\gamma(\mathbf{IJ}|_{i+k})(pq)^{i-1}\right.\\
		&\left.+\sum_{i=k_{1}-k}^{k_{1}-k+\ell(\Upsilon_{\gamma}(\mathbf{IJ})_{k_{1},k_{2}})-1}\gamma(\mathbf{IJ}|_{i+k})(pq)^{i-1}\right).
	\end{split}
\end{align}
For convenience, we define
\begin{equation*}
	\xi_{1}=\frac{\xi+\gamma^{*}(\mathbf{I})}{(pq)^{k}},\quad \xi_{2}=\frac{\xi_{1}+\sum_{i=1}^{\ell(\Upsilon_{\gamma}(\mathbf{IJ})_{k_{0},k_{1}})}\gamma(\mathbf{IJ}|_{i+k})(pq)^{i-1}}{(pq)^{k_{1}-k-1}},
\end{equation*}
and
\begin{equation*}
	\zeta_{1}=\sum_{i=1}^{\ell(\Upsilon_{\gamma}(\mathbf{IJ})_{k_{0},k_{1}})}\gamma(\mathbf{IJ}|_{i+k})(pq)^{i-1},\quad \zeta_{2}=\sum_{i=1}^{\ell(\Upsilon_{\gamma}(\mathbf{IJ})_{k_{1},k_{2}})}\gamma(\mathbf{IJ}|_{k_{1}+i-1})(pq)^{i-1}.
\end{equation*}
We then obtain the following inequality:
\begin{align}\label{eq3.3}
	\begin{split}
		\left|\widehat{\mu}_{Q,D}(\xi_{1}+\zeta_{1}+(pq)^{k_{1}-k-1}\zeta_{2})\right|
		&=\left|\widehat{\mu}_{k_{1}-k-1}(\xi_{1}+\zeta_{1}+(pq)^{k_{1}-k-1}\zeta_{2})\right|
		\left|\widehat{\mu}_{Q,D}(\xi_{2}+\zeta_{2})\right|\\
		&=|\widehat{\mu}_{k_{1}-k-1}(\xi_{1}+\zeta_{1})|
		|\widehat{\mu}_{Q,D}(\xi_{2}+\zeta_{2})|\\
		&\ge|\widehat{\mu}_{Q,D}(\xi_{1}+\zeta_{1})|
		|\widehat{\mu}_{Q,D}(\xi_{2}+\zeta_{2})|.
	\end{split}
\end{align}
Note that $\gamma(\mathbf{IJ}|_{k+1}) = \gamma(\mathbf{I}0) = e_{\mathbf{I}} \in B_{q} \setminus \{\bm{0}\}$.
As $B_{q} \setminus \{\bm{0}\} \cap \mathcal{Z}(\widehat{\delta}_{Q^{-1}D}) = \emptyset$, we have $\zeta_1 := (\zeta_{1,1}, \zeta_{2,1}, \dots, \zeta_{n,1})^t \notin \mathcal{Z}(\widehat{\mu}_{Q,D})$. Let $\xi_1 := (\xi_{1,1}, \xi_{2,1}, \dots, \xi_{n,1})^t$.
From $\gamma(\mathbf{IJ}|_{j}) \in \Gamma_{q}$ and $\xi \in [-1,1]^n$, it follows that
\begin{equation*}
	|\zeta_{i,1}| \le \frac{pq}{2} \left(1 + pq + \dots + (pq)^{\ell(\Upsilon_{\gamma}(\mathbf{IJ})_{k_{0},k_{1}})-1}\right) \le (pq)^{\ell(\Upsilon_{\gamma}(\mathbf{IJ})_{k_{0},k_{1}})}
\end{equation*}
and since $pq\geq 3$,
\begin{equation*}
	|\xi_{i,1}| \le \frac{1 + \frac{pq}{2}(1 + pq + \dots + (pq)^{k-1})}{(pq)^k} =\frac{(pq)^{k+1}+pq-2}{2(pq-1)(pq)^{k}}\leq \frac{(pq)^{k+1}+pq-2}{\frac{4pq}{3}(pq)^{k}}\leq \frac{5}{6}
\end{equation*}
for $i = 1,2,\dots,n$. Similarly, we can show that $\xi_2 \in [-\frac{5}{6}, \frac{5}{6}]^n$, $\zeta_2 \notin \mathcal{Z}(\widehat{\mu}_{Q,D})$, and $\zeta_2 \in \mathbb{Z}^n \cap [-(pq)^{\ell(\Upsilon_{\gamma}(\mathbf{IJ})_{k_{0},k_{1}})}, (pq)^{\ell(\Upsilon_{\gamma}(\mathbf{IJ})_{k_{1},k_{2}})}]^n$. Then, by the continuity of $\widehat{\mu}_{Q,D}(\xi)$, there exists a constant $M_{\Upsilon}$ such that
$$\eta = \min_{\xi\in H}\{ | \widehat{\mu}_{Q,D}(\xi) | \} > 0,$$
where
\begin{equation*}
	H = \bigcup\left\{ \zeta + \left[-\frac{5}{6}, \frac{5}{6}\right]^n : \zeta \in \mathbb{Z}^n \cap [-M_{\Upsilon}, M_{\Upsilon}]^n, \zeta \notin \mathcal{Z}(\widehat{\mu}_{Q,D}) \right\}.
\end{equation*}
Therefore, combining \eqref{eq3.2} and \eqref{eq3.3}, we have
\begin{equation*}
	| \widehat{\mu}_{Q,D}(\xi + \gamma^{*}(\mathbf{IJ})) | \ge \eta^2 | \widehat{\mu}_{k}(\xi + \gamma^{*}(\mathbf{I})) |.
\end{equation*}
Furthermore, from \eqref{eq3.1}, we deduce that
\begin{equation*}
	| \widehat{\mu}_{Q,D}(\xi + \gamma^{*}(\mathbf{IJ})) | \ge \eta | \widehat{\mu}_{k}(\xi + \gamma^{*}(\mathbf{IJ})) | \ge \eta^2 | \widehat{\mu}_{k}(\xi + \gamma^{*}(\mathbf{IJ})) |.
\end{equation*}
In general, if $j \leq M_{\Upsilon}$, then
\begin{equation*}
	| \widehat{\mu}_{Q,D}(\xi + \gamma^{*}(\mathbf{IJ})) | \ge \eta^{j+1} | \widehat{\mu}_{k}(\xi + \gamma^{*}(\mathbf{IJ})) | \ge \eta^{M_{\Upsilon}+1} | \widehat{\mu}_{k}(\xi + \gamma^{*}(\mathbf{IJ})) |.
\end{equation*}
We set $c = \max\{ \eta^{-2}, \eta^{-M_{\Upsilon}-1} \}$, which completes the proof of the desired inequality.
\end{proof}
We now prove Theorem \ref{thm1.2} using the preceding two lemmas.
\begin{proof}[\textbf{Proof of Theorem \ref{thm1.2}}]
	Let $\gamma$ be a maximal tree mapping and $\Lambda = \gamma^{*}(\Sigma_{p}^{\gamma})$.
	Using the orthogonality of $\Lambda$ and Lemma \ref{lem2.2}, it follows that
	\begin{equation*}
		Q_{\mu,\Lambda}(\xi) = \sum_{i=1}^{\infty} \sum_{\substack{\mathbf{I} \in \Sigma_{p}^{\gamma} \\ \ell(\Upsilon_{\gamma}(\mathbf{I})) = i}} | \widehat{\mu}_{Q,D}(\xi + \gamma^{*}(\mathbf{I})) |^2 \leq 1.
	\end{equation*}
	For $\xi$ with $0 < \|\xi\| < 1$ and any $0 < \varepsilon < 1$, there exists an $N = N(\xi, \varepsilon)$ such that
	\begin{equation}\label{eq3.34}
		\sum_{\substack{\mathbf{I} \in \Sigma_{p}^{\gamma} \\ \ell(\Upsilon_{\gamma}(\mathbf{I})) > N}} | \widehat{\mu}_{Q,D}(\xi + \gamma^{*}(\mathbf{I})) |^2 \leq \varepsilon.
	\end{equation}
Let $\Omega_{N} := \{ \mathbf{I} \in \Sigma_{p}^{\gamma} : \ell(\Upsilon_{\gamma}(\mathbf{I})) \leq N \}$.
For any $\mathbf{I} \in \Omega_{N}$, we have $\mathbf{I}|_{N+1,\infty} = 0^\infty$, which implies that $\Omega_{N} \subset \Sigma_{p}^{N}0^\infty$.
Thus, $\#\Omega_{N} \leq p^N$. For $k > N$, we observe that
\begin{equation}\label{3.6}
	Q_{\mu,\Lambda}(\xi) \geq \sum_{\mathbf{I} \in \Omega_{N}} | \widehat{\mu}_{Q,D}(\xi + \gamma^{*}(\mathbf{I})) |^2 = \sum_{\mathbf{I} \in \Omega_{N}} | \widehat{\mu}_{k}(\xi + \gamma^{*}(\mathbf{I})) |^2 \left| \widehat{\mu}_{Q,D}\left( \frac{\xi + \gamma^{*}(\mathbf{I})}{(pq)^k} \right) \right|^2.
\end{equation}
For each $\mathbf{I} \in \Omega_{N}$,
\begin{equation*}
	\left\| \frac{\xi + \gamma^{*}(\mathbf{I})}{(pq)^k} \right\| \leq \frac{1 + (\sqrt{n}pq/2)\sum_{t=0}^{N-1}(pq)^t}{(pq)^k} < \frac{\sqrt{n}}{(pq - 1)(pq)^{k-N-1}}.
\end{equation*}
By continuity of $\widehat{\mu}_{Q,D}$ and $\widehat{\mu}_{Q,D}(\bm{0}) = 1$, there exists a sufficiently large $k$ such that
\begin{equation*}
	\left| \widehat{\mu}_{Q,D}\left( \frac{\xi + \gamma^{*}(\mathbf{I})}{(pq)^k} \right) \right|^2 > 1 - \varepsilon
\end{equation*}
for all $\mathbf{I} \in \Omega_{N}$.
Hence, by \eqref{1.3},
\begin{equation}\label{eq3.5}
	Q_{\mu,\Lambda}(\xi) \geq (1 - \varepsilon) \sum_{\mathbf{I} \in \Omega_{N}} | \widehat{\mu}_{k}(\xi + \gamma^{*}(\mathbf{I})) |^2
\end{equation}
for sufficiently large $k$.
By Lemmas \ref{lem2.2} and \ref{lem3.1},
\begin{equation*}
	1 = \sum_{\mathbf{I} \in \Sigma_{p}^{k}} | \widehat{\mu}_{k}(\xi + \gamma^{*}(\mathbf{I}))|^2 = \sum_{\mathbf{I} \in \Omega_{N}} | \widehat{\mu}_{k}(\xi + \gamma^{*}(\mathbf{I})) |^2 + \sum_{\mathbf{I} \in \Sigma_{p}^{k} \setminus \Omega_{N}|_{k}} | \widehat{\mu}_{k}(\xi + \gamma^{*}(\mathbf{I})) |^2
\end{equation*}
for $k \geq N$.
Evidently, $\#( \Omega_{N}|_{k} ) = \#\Omega_{N}$. Then,
\begin{equation}\label{3.9}
		Q_{\mu,\Lambda}(\xi) \geq (1 - \varepsilon)\left(1-\sum_{\mathbf{I} \in \Sigma_{p}^{k} \setminus \Omega_{N}|_{k}} | \widehat{\mu}_{k}(\xi + \gamma^{*}(\mathbf{I})) |^2\right) .
\end{equation}
For any $\mathbf{I} \in \Sigma_{p}^{k} \setminus\Omega_{N}|_{k}$, there exists a $\mathbf{J} \in \Sigma_{p}^{\infty}$ such that $\mathbf{IJ} \in \Sigma_{p}^{\gamma}$ and $\ell(\Upsilon_{\gamma}(\mathbf{IJ})) > N$.
By \eqref{eq3.34} and Lemma \ref{lem3.2},
\begin{align*}
	1-\sum_{\mathbf{I} \in \Sigma_{p}^{k} \setminus\Omega_{N}|_{k}} | \widehat{\mu}_{k}(\xi + \gamma^{*}(\mathbf{I})) |^2
	&= 1-\sum_{\mathbf{I} \in \Sigma_{p}^{k} \setminus\Omega_{N} |_{k}} | \widehat{\mu}_{k}(\xi + \gamma^{*}(\mathbf{IJ}))|^2 \\
	&\geq 1-c \sum_{\mathbf{I} \in \Sigma_{p}^{k} \setminus \Omega_{N} |_{k}}| \widehat{\mu}_{Q,D}(\xi + \gamma^{*}(\mathbf{IJ})) |^2 \\
	&\geq 1-c \sum_{\substack{\mathbf{I} \in \Sigma_{p}^{\gamma} \\ \ell(\Upsilon_{\gamma}(\mathbf{I})) > N}}| \widehat{\mu}_{Q,D}(\xi + \gamma^{*}(\mathbf{I}))|^2 > 1-c\varepsilon.
\end{align*}
Finally, combining \eqref{eq3.5} and \eqref{3.9}, for sufficiently large $k$, we find $Q_{\mu,\Lambda}(\xi) \geq (1 - \varepsilon)(1 - c\varepsilon)$.
Sending $\varepsilon \to 0$, we conclude that $Q_{\mu,\Lambda}(\xi) = 1$. By Lemma \ref{lem2.2},  $\Lambda=\gamma^{*}(\Sigma_{p}^{\gamma})$ is a spectrum of $\mu_{Q,D}$.
\end{proof}
\begin{proof}[\textbf{Proof of Corollary} \ref{col1.5}]
Note that
\begin{equation*}
	\Lambda(Q, C_{q}) := \left\{ \sum_{k=0}^{n} Q^k C_{q} : n \geq 1 \right\}
	= \sum_{k=0}^{\infty} Q^k C_{q},
\end{equation*}
where $C_{q} = A_{q} \pmod{Q}$. We then only need to set $e_{\mathbf{I}} = \bm{0}$ for all $\mathbf{I} \in \Sigma_{p}^{*}$ (as defined in Definition \ref{defn1.3}) to construct $\gamma$.
For each $\mathbf{I} \in \Sigma_{p}^{*}$, we take $\mathbf{J}_{\mathbf{I}} = 0^\infty$, which implies $\Theta_{\mathbf{I}}(\mathbf{J}_{\mathbf{I}}) = 0$.
The proof is thus complete by virtue of Theorem \ref{thm1.2}.
\end{proof}
\section{ \bf{Spectral eigenmatrices in integer matrix}}
In this section, we prove Theorem \ref{thm1.3}, which concerns integer eigenmatrices $R \in M_n(\mathbb{Z})$ of $\mu_{Q,D}$.
First, we present a simple lemma that characterizes the properties of the maximal orthogonal set of $\mu_{Q,D}$.
\begin{lemma}\label{lem4.1}
	If \(\bm{0} \in\Lambda\) is a maximal orthogonal set of \(\mu_{Q,D}\), then there exist \(\lambda_{1}, \lambda_{2}, \dots, $ $\lambda_{p-1} \in \Lambda\) such that \(\{\lambda_{1}, \lambda_{2}, \dots, \lambda_{p-1}\} \equiv \{q\bm{a}, 2q\bm{a}, \ldots, (p-1)q\bm{a}\} \pmod{Q}\).
\end{lemma}
\begin{proof}
	By Theorem \ref{thm1.1}, there exists a maximal mapping $\gamma: \Sigma_{p}^{*} \to \Gamma_{q}$ such that $\Lambda = \gamma^{*}(\Sigma_{p}^{\gamma})$.
	By Definition \ref{defn1.3}, there exist $\mathbf{I}_{1}, \mathbf{I}_{2}, \dots, \mathbf{I}_{p-1} \in \Sigma_{p}^{*}$ for which $1\mathbf{I}_{1}, 2\mathbf{I}_{2}, \dots, (p-1)\mathbf{I}_{p-1} \in \Sigma_{p}^{\gamma}$.
	Hence, $\gamma^{*}(1\mathbf{I}_{1}), \gamma^{*}(2\mathbf{I}_{2}), \dots, \gamma^{*}((p-1)\mathbf{I}_{p-1}) \in \Lambda$, and
	\begin{equation*}
		\gamma^{*}(i\mathbf{I}_{i}) \equiv \gamma(i) \equiv iq\bm{a} \pmod{Q}, \quad i = 1, 2, \dots, p-1.
	\end{equation*}
	This completes the proof.
\end{proof}
The following lemma establishes the necessity of Theorem \ref{thm1.3}.
 \begin{lemma}\label{lem4.2}
 Let $R\in M_{n}(\mathbb{Z})$ and let $\bm{0}\in \Lambda$ be a maximal orthogonal set of $\mu_{Q,D}$. Then the following holds.
 \begin{enumerate}
 	\item[\textbf{\rm(i)}] If $R\Lambda$ is a maximal orthogonal set of $\mu_{Q,D}$, then there exists $k\in\{1,2,\dots,p-1\}$ such that $R\bm{a}=k\bm{a}\pmod {p\mathbb{Z}^{n}}$;
 	\item[\textbf{\rm(ii)}] If there exists $k\in\{1,2,\dots,p-1\}$ such that $R\bm{a}=k\bm{a}\pmod {p\mathbb{Z}^{n}}$, then $R\Lambda$ is an orthogonal set of $\mu_{Q,D}.$

 \end{enumerate}
 \end{lemma}
 \begin{proof}(i) By Lemma \ref{lem4.1}, if $R\Lambda$ is a maximal orthogonal set of $\mu_{Q,D}$, then there exists $\lambda\in\Lambda$ such that
 	\begin{equation*}
 		R\lambda \equiv q\bm{a} \pmod{Q}.
 	\end{equation*}
 	Since $\lambda\in\Lambda\setminus\{\bm{0}\}\subset\mathcal{Z}(\widehat{\mu}_{Q,D})$ and $R\in M_{n}(\mathbb{Z})$,
 	it follows that there exists $k\in\{1,2,\dots,p-1\}$ such that
 	\[
 	\lambda \equiv kq\bm{a} \pmod{Q}.
 	\]
From the above discussion, we have
\begin{equation*}
	kqR\bm{a} \equiv R\lambda \equiv q\bm{a} \pmod{Q}.
\end{equation*}
There thus exists $k_1 \in \{1, 2, \dots, p-1\}$ such that $kk_1 \equiv 1 \pmod{p}$.
Furthermore,
\begin{equation*}
	qR\bm{a} \equiv k_1 q\bm{a} \pmod{Q}.
\end{equation*}
This implies that $R\bm{a} \equiv k_1 \bm{a} \pmod{p\mathbb{Z}^{n}}$. 	
\item [\textbf{\rm(ii)}]
Since $\Lambda$ is an orthogonal set for $\mu_{Q,D}$, for any $\lambda_1, \lambda_2 \in \Lambda$, there exists $j\in\mathbb{Z}$ such that $\lambda_1 - \lambda_2 \in \mathcal{Z}(\widehat{\delta}_{Q^{-j}D})$.
Hence,
\begin{equation*}
	R(\lambda_1 - \lambda_2) \in RQ^j \mathcal{Z}(\widehat{\delta}_D) = Q^j \bigcup_{l=1}^{p-1} \left( \frac{l}{p} R\bm{a} + R\mathbb{Z}^n \right).
\end{equation*}
By assumption, there exists $k\in\mathbb{Z}$ such that $R\bm{a} \equiv k\bm{a} \pmod{p\mathbb{Z}^n}$.
This implies that
\begin{equation*}
	R(\lambda_1 - \lambda_2) \in Q^j \bigcup_{l=1}^{p-1} \left( \frac{lk}{p}\bm{a} + \mathbb{Z}^n \right) \subset \mathcal{Z}(\widehat{\mu}_{Q,D}).
\end{equation*}
Thus, $(R\Lambda - R\Lambda) \setminus\{\bm{0}\} \subset \mathcal{Z}(\widehat{\mu}_{Q,D})$, meaning that $R\Lambda$ is an orthogonal set for $\mu_{Q,D}$.
\end{proof}
In what follows, we introduce the concept of a cut set and use it to establish the sufficiency of Theorem \ref{thm1.3}.
\begin{defn}
	Let $\mathcal{A} \subset \Sigma_{p}^{*}$. The set $\mathcal{A}$ is called a \textit{cut set} if every $\mathbf{I} \in \Sigma_{p}^{\infty}$ has a unique prefix in $\mathcal{A}$. That is, for every $\mathbf{I} \in \Sigma_{p}^{\infty}$, there exists a unique integer $k$ such that $\mathbf{I}|_{k} \in \mathcal{A}$.
\end{defn}
In the figure below, we give an intuitive example of two cut sets with $p=3$, where
$$\mathcal{A}_{\text{blue}} = \{00,01,02,10,11,12,20,21,22\} = \Sigma_{3}^{2}$$
and
$$\mathcal{A}_{red}=\{000,001,002,010,011,012,02,10,11,12,200,201,202,21,22\}.$$
\par % 换行（可选，若要图片在语句下一行）
\noindent % 取消首行缩进
\begin{minipage}{\textwidth}
	\centering
	% 第一张图
	\begin{minipage}{0.50\textwidth}
		\centering
		\includegraphics[width=\textwidth]{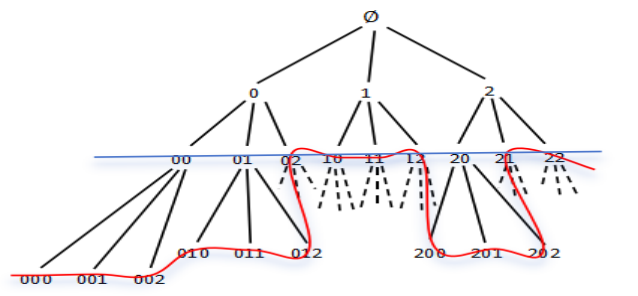}
		\footnotesize Figure 2. Cut set: Red line and blue line
		\label{tree2}
	\end{minipage}
\end{minipage}
\medspace\\
It is known that $\mathcal{A}$ is a finite set if $\mathcal{A}$ is a cut set, and $\mathcal{A}^*$ can be defined in much the same way as $\Sigma_p^*$.
 \begin{lemma}\label{lem4.3}
 	Let $\mathcal{A}$ be a cut set of $\Sigma_{p}^{*}$. Then, $\Sigma_{p}^{*}0^{\infty} = \mathcal{A}^{*}0^{\infty}$.
  \end{lemma}
 \begin{proof}
 Since $ \mathcal{A}\subset\Sigma_{p}^{*} $, we know that $ \mathcal{A}^{*}0^{\infty}\subset\Sigma_{p}^{*}0^{\infty} .$ On the other hand, for any $\mathbf{I} \in \Sigma_{p}^{*}0^{\infty}$, there exists $n_1 > 0$ such that $\mathbf{I}|_{n_1} := \mathbf{I}_1 \in \mathcal{A}$. Then $\mathbf{I} = \mathbf{I}_1 \mathbf{I}_{n_1+1,\infty}$, where $\mathbf{I}_{n_1+1,\infty} \in \Sigma_{p}^{\infty}$. Applying the same argument to $\mathbf{I}_{n_1+1,\infty}$, there exists $n_2 > n_1$ such that $\mathbf{I}_{n_1+1,\infty} = \mathbf{I}_2 \mathbf{I}_{n_2+1,\infty}$ with $\mathbf{I}_2 \in \mathcal{A}$. Since $\mathbf{I}$ has finite effective length and $0^k \in \mathcal{A}$ for some $k > 1$, an inductive argument yields that
 \begin{equation*}
 	\mathbf{I} = \mathbf{I}_1 \mathbf{I}_2 \dots \mathbf{I}_m (0^k)^\infty \in \mathcal{A}^* 0^\infty,
 \end{equation*}
 where $\ell(\mathbf{I}_m) \neq 0$. This completes the proof.
 \end{proof}
\begin{lemma}\label{thm4.5}
	Let $\mathcal{A} \subset \Sigma_{p}^{*}$ be a cut set. Suppose a maximal map $\gamma$ satisfies
	\begin{equation}\label{4.2}
		\gamma(\mathbf{IJ}) = \gamma(\mathbf{J}) \quad \text{for all } \mathbf{I} \in \mathcal{A}, \mathbf{J} \in \Sigma_{p}^{*}.
	\end{equation}
	Then, $\Sigma_{p}^{\gamma} = \mathcal{A}^{*}0^{\infty}$ and $\Lambda = \gamma^{*}(\Sigma_{p}^{\gamma})$ is a spectrum of $\mu_{Q,D}$.
\end{lemma}
 \begin{proof}
From the definition of $\Sigma_{p}^{\gamma}$ and Lemma \ref{lem4.3}, it follows that
 	\begin{equation*}
 		\Sigma_{p}^{\gamma} \subset \Sigma_{p}^{*}0^{\infty} = \mathcal{A}^{*}0^{\infty}.
 	\end{equation*}
 	On the other hand, \eqref{4.2} implies that
 	\begin{equation*}
 		\gamma(\mathbf{IJ}) = \gamma(\mathbf{J}) \quad \text{for all } \mathbf{I} \in \mathcal{A}^{*}, \mathbf{J} \in \Sigma_{p}^{*}.
 	\end{equation*}
Then, for any word $\mathbf{I} \in \mathcal{A}^{*}$ of length $l_1$ and for any $l_1 < l_2$,
 	\begin{equation*}
 		\gamma(\mathbf{I}0^{\infty}|_{l_2}) = \gamma(\mathbf{I}0^{l_2-l_1}) = \gamma(0^{l_2-l_1}) = \mathbf{0}.
 	\end{equation*}
 	Thus, $\mathcal{A}^{*}0^{\infty} \subset \Sigma_{p}^{\gamma}$, and hence $\mathcal{A}^{*}0^{\infty} = \Sigma_{p}^{\gamma}$. Next, we use Theorem \ref{thm1.2} to show that $\Lambda = \gamma^{*}(\Sigma_{p}^{\gamma})$ is a spectrum of $\mu_{Q,D}$. Denote $N := \max_{\mathbf{I} \in \mathcal{A}} \ell(\mathbf{I})$.
 	For any $\mathbf{I} \in \Sigma_{p}^{*}$, since $\mathbf{I}0^\infty \in \Sigma_{p}^{*}0^\infty = \mathcal{A}^{*}0^\infty$, we can find an integer $0 \le k < N$ such that $\mathbf{I}0^k \in \mathcal{A}^{*}$. Then, \eqref{4.2} implies that $\gamma(\mathbf{I}0^\infty|_{j}) = \bm{0}$ for all sufficiently large integers $j$. Then, setting $\mathbf{J}_{\mathbf{I}} = 0^\infty$, we can deduce that $\Theta_{\mathbf{I}}(\mathbf{J}_{\mathbf{I}}) \le k < N$, and by Theorem \ref{thm1.2}, it follows that $\Lambda = \gamma^{*}(\Sigma_{p}^{\gamma})$ is a spectrum of $\mu_{Q,D}$.
\end{proof}
In order to prove the sufficiency of Theorem \ref{thm1.3}, we choose a suitable cut set $\mathcal{A} = \Sigma_{p}^{k}$. If a maximal mapping $\gamma$ satisfies
\begin{equation*}
	\gamma(\mathbf{IJ}) = \gamma(\mathbf{J}) \quad \text{for all } \mathbf{I} \in \Sigma_{p}^{k}, \mathbf{J} \in \Sigma_{p}^{*},
\end{equation*}
then from Lemma \ref{thm4.5}, it follows that $\Sigma_{p}^{\gamma} = (\Sigma_{p}^{k})^{*}0^\infty = \Sigma_{p}^{k}\Sigma_{p}^{\gamma}$. Furthermore, we have
\begin{align*}
	\gamma^{*}(\Sigma_{p}^{\gamma}) &= \gamma^{*}(\Sigma_{p}^{k}) + Q^{k}\gamma^{*}(\Sigma_{p}^{\gamma}) = \dots = \gamma^{*}(\Sigma_{p}^{k}) + Q^{k}\gamma^{*}(\Sigma_{p}^{\gamma}) + \dots + Q^{nk}\gamma^{*}(\Sigma_{p}^{\gamma}) + \dots \\
	&= \Lambda(Q^{k}, \gamma^{*}(\Sigma_{p}^{k})).
\end{align*}
Let
\begin{equation*}
	D_{k} = D + QD + \dots + Q^{k-1}D.
\end{equation*}
It follows from \eqref{eq1.2} that
\begin{equation*}
	\mu_{Q,D} = \delta_{Q^{-k}D_{k}} * \delta_{Q^{-2k}D_{k}} * \dots = \mu_{Q^{k},D_{k}}.
\end{equation*}
From Lemma \ref{lem3.1}, it is easy to verify that $\gamma^{*}(\Sigma_{p}^{k})$ is a spectrum of $\delta_{Q^{-k}D_{k}}$. This further implies that the triple $(Q^{k}, D_{k}, \gamma^{*}(\Sigma_{p}^{k}))$ is a Hadamard triple.
 \begin{lemma}\label{lem4.6}
 	Assume there exists $l \in \{1, 2, \dots, p-1\}$ such that $R\bm{a} = l\bm{a} \pmod{p\mathbb{Z}^n}$. Let $\gamma$ be a maximal mapping satisfying
 	\begin{equation}\label{eq4.2}
 		\gamma(\mathbf{IJ}) = \gamma(\mathbf{J}) \quad \text{for all } \mathbf{I} \in \Sigma_{p}^{k}, \mathbf{J} \in \Sigma_{p}^{*}.
 	\end{equation}
 	Then, $(Q^k, D_k, R\gamma^{*}(\Sigma_{p}^{k}))$ is a Hadamard triple.
\end{lemma}
\begin{proof}
For any $\mathbf{I}=i_{1}i_{2}\dots i_{k}\in\Sigma_{p}^{k},$
\begin{equation*}
	\gamma^{*}(\mathbf{I})=\sum_{j=1}^{k}(pq)^{j-1}\gamma(\mathbf{I}|_{j})=\sum_{j=1}^{k}(pq)^{j-1}(e_{\mathbf{I}|_{j-1}}+i_{j}q\bm{a}\pmod Q).
\end{equation*}
Since $\gamma^{*}(\Sigma_{p}^{k})$ is a spectrum of $\delta_{Q^{-k}D_{k}}$, for any two distinct $\mathbf{I}, \mathbf{J} \in \Sigma_{p}^{k}$, let $1 \le t \le k$ be the smallest index such that $\mathbf{I}|_{t} \neq \mathbf{J}|_{t}$. Then
\begin{align*}
	\gamma^{*}(\mathbf{I}) - \gamma^{*}(\mathbf{J}) &= (pq)^{t-1}(i_t - j_t)q\bm{a} \pmod{Q} + \sum_{m=t+1}^{\infty} (pq)^{m-1}(\gamma(\mathbf{I}|_{m}) - \gamma(\mathbf{J}|_{m})) \\
	&:= (pq)^{t-1}\omega \in \mathcal{Z}(\widehat{\delta}_{Q^{-k}D_{k}}) = \bigcup_{j=1}^{k} (pq)^{j-1} \bigcup_{j=1}^{p-1} (jq\bm{a} + pq\mathbb{Z}^n),
\end{align*}
where $\omega \in \cup_{j=1}^{p-1} (jq\bm{a} + pq\mathbb{Z}^n)$. Therefore, by the hypothesis of the lemma,
\begin{equation*}
	R\gamma^{*}(\mathbf{I}) - R\gamma^{*}(\mathbf{J}) = (pq)^{t-1}R\omega \in \mathcal{Z}(\widehat{\delta}_{Q^{-k}D_{k}}).
\end{equation*}
This establishes that $R\gamma^{*}(\Sigma_{p}^{k})$ is an orthogonal set with respect to $\delta_{Q^{-k}D_{k}}$. Since $R$ is an invertible matrix, the cardinality of $R\gamma^{*}(\Sigma_{p}^{k})$ is $p^k$, which coincides with the dimension of $L^2(\delta_{Q^{-k}D_{k}})$. This confirms that $R\gamma^{*}(\Sigma_{p}^{k})$ is a spectrum of $\delta_{Q^{-k}D_{k}}$, and thus the triple $(Q^k, D_k, R\gamma^{*}(\Sigma_{p}^{k}))$ is a Hadamard triple.
\end{proof}
Combining Lemma \ref{lem2.4} and Lemma \ref{lem4.6}, we can directly derive the following lemma.
\begin{lemma}\label{lem4.7}
Assume there exists $l \in \{1, 2, \dots, p-1\}$ such that $R\bm{a} = l\bm{a} \pmod{p\mathbb{Z}^n}$. Let $\gamma$ be a maximal mapping satisfying
\begin{equation*}
	\gamma(\mathbf{IJ}) = \gamma(\mathbf{J}) \quad \text{for all } \mathbf{I} \in \Sigma_{p}^{k}, \mathbf{J} \in \Sigma_{p}^{*}.
\end{equation*}
Then, $R\Lambda(Q^k, \gamma^{*}(\Sigma_{p}^{k}))$ fails to be a spectrum of $\mu_{Q,D}$ if and only if there exists a nonzero $D_k$-cycle $\{x_0, x_1, \dots, x_{s-1}\}$ contained in $T(Q^k, R\gamma^{*}(\Sigma_{p}^{k})).$
\end{lemma}
Next, we construct a mapping satisfying \eqref{eq4.2} to apply the aforementioned two lemmas, and we immediately show that it is a maximal mapping.
\begin{defn}\label{defn4.8}
	Let $q, k \ge 2$. Define $\gamma_k: \Sigma_{p}^{*} \to \Gamma_q$ to be a mapping satisfying the following:
	\begin{enumerate}
		\item[\rm(i)] $\gamma_k(i_1i_2\dots i_j) = i_j q\bm{a} \pmod{Q}$ for $1 \le j \le k-1$;
		\item[\rm(ii)] $\gamma_k(0^{k-1}i_k) = i_k q\bm{a} \pmod{Q}$, and $\gamma_k(i_1i_2\dots i_k) = e_{\bm{a}} + i_k q\bm{a} \pmod{Q}$ if $i_1i_2\dots i_{k-1} \neq 0^{k-1}$, where $e_{\bm{a}} = j\bm{a} \pmod{p\mathbb{Z}^n} \in B_q$ for some $j \in \{1, 2, \dots, p-1\}$;
		\item[\rm(iii)] $\gamma_k(\mathbf{IJ}) = \gamma_k(\mathbf{J})$ for all $\mathbf{I} \in \Sigma_{p}^{k}$, $\mathbf{J} \in \Sigma_{p}^{*}$.
	\end{enumerate}
\end{defn}

\begin{pro}\label{pro4.9}
	Let $k \ge 2$ and let $\gamma_k$ be the mapping given by Definition \ref{defn4.8}. Then, $\gamma_k$ is a maximal mapping, and for any $i_1i_2\dots i_k \neq 0^k$, the sequence $\gamma_k(i_1)\gamma_k(i_1i_2)\dots\gamma_k(i_1i_2\dots i_k)$ is not a repetition of any word in $\Gamma_q^{*}$; that is, there exists no $\eta \in \Gamma_q^{*}$ with $\eta \neq 0^k$ and no integer $l > 1$ such that
	\begin{equation*}
		\gamma_k(i_1)\gamma_k(i_1i_2)\dots\gamma_k(i_1i_2\dots i_k) = \eta^l.
	\end{equation*}
\end{pro}
\begin{proof}
	By the definition of a maximal mapping, $\gamma_k$ in Definition \ref{defn4.8} is easily seen to be a maximal mapping.  For the second conclusion, if there exist $\eta_0 \in \Gamma_q^{*}$ and an integer $l_0 > 1$ such that
	\begin{equation*}
		\gamma_k(i_1) \gamma_k(i_1i_2) \dots \gamma_k(i_1i_2\dots i_k) = \eta_0^{l_0},
	\end{equation*}
	then there exists an integer $k_0$ such that $k = k_0 l_0$ and
	\begin{equation*}
		\gamma_k(i_1) \gamma_k(i_1i_2) \dots \gamma_k(i_1i_2\dots i_{k_0}) = \eta_0.
	\end{equation*}
	Therefore,
	\begin{equation*}
		\gamma_k(i_1i_2\dots i_{k_0}) = \gamma_k(i_1i_2\dots i_{k_0 l_0}) = \gamma_k(i_1i_2\dots i_k).
	\end{equation*}	
If $i_1i_2\dots i_{k-1} = 0^{k-1}$, then $\gamma_k(i_1i_2\dots i_k) = i_k q\bm{a} \pmod{Q} = \gamma_k(0^{k_0}) = \mathbf{0}$. For the remaining case, $i_{k_0} q\bm{a} = e_{\bm{a}} + i_k q\bm{a} \pmod{Q}$, which is a contradiction as $q \ge 2$.
\end{proof}
\begin{thm}\label{thm4.10}
	Let $\gamma_k$ be the mapping given by Definition \ref{defn4.8}. Assume there exists $l \in \{1, 2, \dots, p-1\}$ such that $R\bm{a} = l\bm{a} \pmod{p\mathbb{Z}^n}$. Then, for sufficiently large $k$, $R\Lambda(Q^k, \gamma_k^{*}(\Sigma_{p}^{k}))$ and $\Lambda(Q^k, \gamma_k^{*}(\Sigma_{p}^{k}))$ are spectra of $\mu_{Q,D}$.
\end{thm}
\begin{proof}
	For any $k \ge 2$, let $\mathcal{A} = \Sigma_{p}^{k}$ be a cut set. By Lemma \ref{thm4.5}, $\Lambda(Q^k, \gamma_k^{*}(\Sigma_{p}^{k}))$ is a spectrum of $\mu_{Q,D}$. For any $\alpha \in T(Q, \Gamma_q) \cap \mathbb{Q}^n$, the representation of $\alpha$ is eventually periodic. Denote by $P(\alpha)$ its minimal period in \eqref{eq2.1}. Define
	\begin{equation*}
		N_R := \max\left\{ P(\alpha) : \alpha \in T(Q, \Gamma_q) \cap R^{-1}\mathbb{Z}^n \right\} < \infty.
	\end{equation*}
In what follows, we prove that $R\Lambda(Q^k, \gamma_k^{*}(\Sigma_{p}^{k}))$ is a spectrum of $\mu_{Q,D}$ for $k > N_R$. Suppose to the contrary that this is not the case; then by Lemma \ref{lem4.7}, there exists a nonzero $D_k$-cycle $\{x_0, x_1, \dots, x_{s-1}\}$ contained in $T(Q^k, R\gamma_k^{*}(\Sigma_{p}^{k})).$ Since
\begin{equation*}
	\left|\widehat{\delta}_{D_k}(\xi)\right| = \prod_{i=1}^{k} \left|\widehat{\delta}_{D}(Q^i \xi)\right|,
\end{equation*}
it follows immediately that $|\widehat{\delta}_{D_k}(\xi)| = 1$ if and only if $\xi \in \mathbb{Z}^n$. Then, all $x_j \in \mathbb{Z}^n$, and
\begin{equation*}
	R^{-1}x_j \in T\left(Q^k, R\gamma_k^{*}(\Sigma_{p}^{k})\right) \cap R^{-1}\mathbb{Z}^n \subset T(Q, \Gamma_q) \cap R^{-1}\mathbb{Z}^n, \quad 0 \le j \le s-1.
\end{equation*}
If $k > P(R^{-1}x_j)$, then by Proposition \ref{pro4.9}, there exist an integer $k_0 < k$ and a word $i_1\dots i_k \in \Sigma_{p}^{k}$ such that
$$\gamma_k(i_1\dots i_{k_0}) = \gamma_k(i_1\dots i_{k_0}\dots i_k) = i_{k_0} q\bm{a} \equiv e_{\bm{a}} + i_k q\bm{a} \pmod{Q}.$$
This is a contradiction, from which it follows that $k \le P(R^{-1}x_j)$. However, as $P(R^{-1}x_j) \le N_R < k$, this contradiction implies that $\{\mathbf{0}\}$ is the only $D_k$-cycle in $T(Q^k, R\gamma_k^{*}(\Sigma_{p}^{k}))$, completing the proof.
\end{proof}
Thus, we have completed the proof of Theorem \ref{thm1.3}. A natural question arising from this result.
\begin{conj}
Let  $ \mu_{Q,D} $ be defined by \eqref{1.1}, where $Q=p\diag[l_{1},l_{2},\dots,l_{n}]\in pM_{n}(\mathbb{Z})$ and $ D $ satisfies  \eqref{1.3} for some prime $  p > 2 $. Then, $R\in M_{n}(\mathbb{Z})$ is a specatral eigenmatrix of the spectral pair $(\mu_{Q,D},\Lambda)$ if and only if there exists $k\in\{1,2,\dots,p-1\}$ such that $R\bm{a}=k\bm{a}\pmod {p\mathbb{Z}^{n}}$.
\end{conj}

\section{ \bf{Spectral eigenmatrices in real diagonal matrix}}

Before proving Theorem \ref{thm1.4}, we introduce some notation and facts. We define $\mu_{Q,D}$ given by \eqref{1.1}, where $Q=p\diag[l_{1},l_{2},\dots,l_{n}]\in pM_{n}(\mathbb{Z})$ with $|l_{1}l_{2}\dots l_{n}|>1$, and the digit set $D$ satisfies \eqref{1.3}. Let
\begin{equation*}
E:=\left\{(i_{1},i_{2},\dots,i_{n})^{t}:i_{1},i_{2},\dots,i_{n}\in[0,p-1]\cap\mathbb{Z}\right\},
\end{equation*}
\begin{equation*}
	S:=\{\alpha_{0},\alpha_{1},\dots,\alpha_{p-1}\}:=\{\bm{0},\bm{a},\dots,(p-1)\bm{a}\}\pmod {p\mathbb{Z}^{n}}\cap E\;\;\text{and}\;\; \tilde{S}:=S\backslash\{\bm{0}\}.
\end{equation*}
Given that $\Lambda$ is a spectrum of $\mu_{Q,D}$ containing $\bm{0}$, we may conclude that $pQ^{-1}\Lambda \subset \mathbb{Z}^{n}$ by virtue of the properties of bi-zero sets. Therefore, for any $\lambda\in pQ^{-1}\Lambda$, there exist $e\in E$ and $\omega\in\mathbb{Z}^{n}$ such that $\lambda=e+m\omega.$ We then decompose $\Lambda$ as follows:
\begin{equation*}
	\Lambda = \bigcup_{e\in E} Q\left(\dfrac{e}{p} + \Omega_{e}\right),
\end{equation*}
where $\Omega_{e} = \{\omega : e + p\omega \in pQ^{-1}\Lambda\}$. Evidently, this union is non-disjoint. If $\Omega_{e} = \emptyset$, then $Q(\frac{e}{p}+\Omega_{e})=\emptyset$.   Since $\bm{0}\in\Lambda$, it follows that $\Omega_{\bm{0}}\neq\emptyset$. Let $\omega_1, \omega_2 \in \Omega_e$ be distinct (i.e., $\omega_1 \neq \omega_2$); then $Q(\frac{e}{p} + \omega_1), Q(\frac{e}{p} + \omega_2) \in \Lambda$. As $\Omega_e \subset \mathbb{Z}^n$, we deduce that
\begin{equation*}
	0 = \widehat{\mu}_{Q,D}(Q(\omega_1 - \omega_2)) = \widehat{\delta}_D(\omega_1 - \omega_2)\widehat{\mu}_{Q,D}(\omega_1 - \omega_2) = \widehat{\mu}_{Q,D}(\omega_1 - \omega_2).
\end{equation*}
Consequently, $\Omega_e$ is an orthogonal set for $\mu_{Q,D}$ whenever $\Omega_e \neq \emptyset$. Next, we demonstrate that
\begin{equation*}
	\Lambda=\bigcup_{e\in E}Q\left(\dfrac{e}{p}+\Omega_{e}\right)=\bigcup_{s\in S}Q\left(\dfrac{s}{p}+\Omega_{s}\right).
\end{equation*}
For any $e\in E$, if $\Omega_{e}\neq\emptyset$, then $Q(\frac{e}{p}+\Omega_{e})\subset\mathcal{Z}(\widehat{\mu}_{Q,D})$. Furthermore,
\begin{equation*}
	e+p\Omega_{e}\subset p\left(\mathcal{Z}(\widehat{\delta}_{D})\bigcup\bigcup_{i=1}^{\infty}Q^{i}\mathcal{Z}(\widehat{\delta}_{D})\right)=S+p\mathbb{Z}^{n}.
\end{equation*}
Hence, $e\in S$. On the other hand, if we define $S' = \{s \in S : \Omega_s \neq \emptyset\}$, it suffices to prove that $S' = S$. By virtue of $p\mathcal{Z}(\widehat{\delta}_D) \subset \mathbb{Z}^n$, there exists $\xi_1$ such that $Q^{-1}\xi_1 \in \mathbb{R}^n \setminus \mathbb{Q}^n$, and for any $s \in S$, $ |\widehat{\delta}_{D}(Q^{-1}\xi_{1}+\frac{s}{p})|^{2}>0$. Thus, if \( S' \neq S \), we deduce that \( \sum_{s \in S'} | \widehat{\delta}_D( Q^{-1} \xi_1 + \frac{s}{p} ) |^2 < 1 \), as \( p^{-1}S \) denotes the spectrum of \( \delta_D \). However, Theorem \ref{lem2.2} implies that
\begin{align*}
	1=\sum_{\lambda\in \Lambda}\left|\widehat{\mu}_{Q,D}(\xi_{1}+\lambda)\right|^{2}&=\sum_{s\in S'}\sum_{\omega\in \Omega_{s}}\left|\widehat{\delta}_{D}\left(Q^{-1}\xi_{1}+\frac{s}{p}+\omega\right)\right|^{2}\left|\widehat{\mu}_{Q,D}\left(Q^{-1}\xi_{1}+\frac{s}{p}+\omega\right)\right|^{2}\\
	&=\sum_{s\in S'}\left|\widehat{\delta}_{D}\left(Q^{-1}\xi_{1}+\frac{s}{p}\right)\right|^{2}\sum_{\omega\in \Omega_{s}}\left|\widehat{\mu}_{Q,D}\left(Q^{-1}\xi_{1}+\frac{s}{p}+\omega\right)\right|^{2}\\
	&\le\sum_{s\in S'}\left|\widehat{\delta}_{D}\left(Q^{-1}\xi_{1}+\frac{s}{p}\right)\right|^{2}<1,
\end{align*}
yielding a contradiction. Therefore, \( S' = S \). In summary, we establish that
\begin{equation}\label{eq5.1}
	\Lambda = \bigcup_{s\in S} Q\left(\dfrac{s}{p} + \Omega_s\right).
\end{equation}

We next state two lemmas required for the proof of the main theorem.
\begin{lemma}\label{lem5.1}
For any infinite word $\mathbf{I}=i_{0}i_{1}i_{2}\dots\in\{-1,1\}^{\mathbb{N}},$ the set
\begin{equation*}
	\Lambda(Q,p^{-1}\mathbf{I}QS):=\left\{\sum_{j=0}^{k}p^{-1}i_{j}	Q^{j+1}s_{j}:s_{j}\in S, k\in\mathbb{N}\right\}
\end{equation*}
is a spectrum of $\mu_{Q,D}.$
\end{lemma}
\begin{proof}
For any $x \in T(Q, \pm p^{-1}QS)$, we have that $x$ can be expressed as $x = \sum_{i=0}^{\infty} p^{-1}Q^{-i}s_i$, where $s_i \in \pm S$, $x = (x_1, \dots, x_n)^t$, and $s_i = (s_{1,i}, \dots, s_{n,i})^t$. Since $|l_1 l_2 \dots l_n| > 1$, it follows that there exists $j \in \{1, \dots, n\}$ such that $|l_j| \ge 2$. We then have
\begin{equation*}
	|x_{j}|\le p^{-1}\sum_{i=0}^{\infty}\left|(pl_{j})^{-i}s_{j,i}\right|\le\frac{(p-1)|l_{j}|}{|1-pl_{j}|}<1.
\end{equation*}
However,
$\mathcal{Z}(\widehat{\mu}_{Q,D})\subset\mathbb{Z}^{n}\backslash\{\bm{0}\}$, which implies
$ T(Q,\pm p^{-1}QS)\cap\mathcal{Z}(\widehat{\mu}_{Q,D})=\emptyset$. Furthermore, since \( \mathcal{Z}(\widehat{\mu}_{Q,D}) \) is a uniformly discrete set and \( T(Q, \pm p^{-1}QS) \) is a compact set, these two sets are separated by a uniform positive distance. By Lemma \ref{lem2.3}, we have that $(Q, D, p^{-1}QS)$ and $(Q, D, -p^{-1}QS)$ are Hadamard triples. Hence, Lemma \ref{lem2.7} yields that $\Lambda(Q, p^{-1}\mathbf{I}QS)$ is a spectrum for $\mu_{Q,D}$.
\end{proof}
\begin{lemma}\label{lem5.2}
Suppose $R$ is an integer diagonal matrix and there exists an integer $k \in \{1, \dots, p-1\}$ such that $R\bm{a} = k\bm{a} \pmod{p\mathbb{Z}^n}$. Then
\begin{enumerate}
	\item[\rm{(i)}]The expansion of $x\in T(Q,\pm p^{-1}QRS)$ is unique;
	\item[\rm{(ii)}]The expansion of
	$x\in T(Q,\pm p^{-1}QRS)\cap\mathcal{Z}(\widehat{\mu}_{Q,D})$ cannot be finite.
\end{enumerate}
\end{lemma}
\begin{proof}
$\rm{(i)}$ For any $x \in T(Q, \pm p^{-1}QRS)$, if $x$ does not have a unique expansion, there exist two distinct vector sequences $\{s_j = (s_{1,j}, s_{2,j}, \dots, s_{n,j})^t\}_{j=0}^{\infty}$ and $\{s_j' = (s_{1,j}', s_{2,j}', \dots, s_{n,j}')^t\}_{j=0}^{\infty} \subset \pm S$ such that
\begin{equation}\label{eq5.2}
	x = p^{-1}\sum_{j=0}^{\infty}Q^{-j}Rs_j = p^{-1}\sum_{j=0}^{\infty}Q^{-j}Rs_j'.
\end{equation}
Let $t = \min\{j \ge 0 : s_j \neq s_j'\}$. As $R\bm{a} = k\bm{a} \pmod{p\mathbb{Z}^n}$, $R$ is invertible. From \eqref{eq5.2}, we obtain
\begin{equation*}
	\left(s_{1,t}-s_{1,t}',\dots,s_{n,t}-s_{n,t}'\right)^{t}=\left(\sum_{j=1}^{\infty}\frac{s_{1,t+j}'-s_{1,t+j}}{(pl_{1})^{j}},\dots,\sum_{j=1}^{\infty}\frac{s_{n,t+j}'-s_{n,t+j}}{(pl_{n})^{j}}\right)^{t}.
\end{equation*}
Since $|l_1 l_2 \dots l_n| > 1$, there exists $i \in \{1, 2, \dots, n\}$ such that $|l_i| \ge 2$, and
\begin{equation*}
	\left|s_{i,t}-s_{i,t}'\right|=\left|\sum_{j=1}^{\infty}\frac{s_{i,t+j}'-s_{i,t+j}}{(pl_{i})^{j}}\right|\le \frac{2p-2}{|pl_{i}|-1}<1.
\end{equation*}
This contradicts $|s_{i,t} - s_{i,t}'| \in \{1, 2, \dots, 2(p-1)\}$. Therefore, the expansion is unique.\\
$\rm{(ii)}$
Since $x \in T(Q, \pm p^{-1}QRS) \cap \mathcal{Z}(\widehat{\mu}_{Q,D})$, there exist $k_1 \in \mathbb{Z}^n$, $j \in \mathbb{N}$, and $k \in \{1, \dots, p-1\}$ such that $x = Q^j(\frac{k\bm{a}}{p} + k_1)$. Suppose there exists $h \in \mathbb{N}$ such that
$$x = p^{-1}\left(Rs_1 + Q^{-1}Rs_2 + Q^{-2}Rs_3 + \dots + Q^{-h}Rs_{h+1}\right),$$
where $s_i \in \pm S$. Then,
\begin{equation}\label{eq5.3}
	\left(Q^hRs_1 + Q^{h-1}Rs_2 + Q^{h-2}Rs_3 + \dots + Q^{h-h}Rs_{h+1}\right) = Q^{j+h}\left(k\bm{a} + pk_1\right).
\end{equation}
The right-hand side of \eqref{eq5.3} belongs to $p\mathbb{Z}^n$, and the left-hand side belongs to $\tilde{S} + p\mathbb{Z}^n$, which is a contradiction. Hence, the proof is complete.
\end{proof}
\begin{lemma}\label{lem5.3}
	Let $\{R_k\}_{k=1}^N$ be a finite sequence of integer diagonal matrices. For each $i \in \{1, \dots, N\}$, there exists $k_i \in \{1, \dots, p-1\}$ such that $R_i\bm{a} = k_i\bm{a} \pmod{p\mathbb{Z}^n}$. Then there exists a spectrum $\Lambda$ for $\mu_{Q,D}$ such that $R_i\Lambda$ is a spectrum for $\mu_{Q,D}$ for every $i \in \{1, \dots, N\}$.
\end{lemma}
\begin{proof}
	We prove this by considering two cases.
	If \( \cup_{k=1}^{N} T(Q, \pm p^{-1}QR_k S) \cap \mathcal{Z}(\widehat{\mu}_{Q,D}) = \emptyset \), then by adapting the proof technique of Lemma \ref{lem5.1}, we conclude that for any infinite word \( \mathbf{I} = i_0 i_1 \dots \in \{-1, 1\}^{\mathbb{N}} \), \( \Lambda(Q, p^{-1} \mathbf{I} QR_k S) \) is a spectrum for \( \mu_{Q,D} \). As \( R_k Q = Q R_k \), we have \( \Lambda(Q, p^{-1}\mathbf{I}QR_k S) = R_k \Lambda(Q, p^{-1}\mathbf{I}QS) \), where \( \Lambda(Q, p^{-1}\mathbf{I}QS) \) is also a spectrum for \( \mu_{Q,D} \).
	
	If $\cup_{k=1}^{N} T(Q, \pm p^{-1}QR_k S) \cap \mathcal{Z}(\widehat{\mu}_{Q,D})\neq\emptyset$, then since $\cup_{k=1}^{N} T(Q, \pm p^{-1}QR_k S)$ is compact and $\mathcal{Z}(\widehat{\mu}_{Q,D})$ is uniformly discrete, their intersection is finite. We define
	\begin{equation*}
		\Delta_{k,1} := \left\{ p^{-1}\sum_{j=0}^{\infty}i_j Q^{-j}R_k s_j \in \mathcal{Z}(\widehat{\mu}_{Q,D}) : \exists\ \text{infinitely many}\ j \in \mathbb{N}\ \text{with}\ i_j|s_j| > 0 \right\}
	\end{equation*}
	and
	\begin{equation*}
		\Delta_{k,2} := \left\{ p^{-1}\sum_{j=0}^{\infty}i_j Q^{-j}R_k s_j \in \mathcal{Z}(\widehat{\mu}_{Q,D}) \setminus \Delta_{k,1} : \exists\ \text{infinitely many}\ j \in \mathbb{N}\ \text{with}\ i_j|s_j| < 0 \right\}.
	\end{equation*}
	From Lemma \ref{lem5.2} $\rm(ii)$, it follows that
	$$\bigcup_{k=1}^{N} T(Q, \pm p^{-1}QR_k S) \cap \mathcal{Z}(\widehat{\mu}_{Q,D}) = \bigcup_{k=1}^{N} (\Delta_{k,1} \cup \Delta_{k,2}).$$
Without loss of generality, we assume
	\begin{equation*}
		\Delta_{k,1} = \{\beta_{k,1}, \beta_{k,2}, \dots, \beta_{k,t_k}\} \; \text{and} \; \Delta_{k,2} = \{\beta_{k,t_k+1}, \beta_{k,t_k+2}, \dots, \beta_{k,t_k+h_k}\},
	\end{equation*}
	where $t_k, h_k \in \mathbb{N}$ denote the cardinalities of $\Delta_{k,1}$ and $\Delta_{k,2}$, respectively. For each $1 \le k \le N$ and $1 \le i \le t_k + h_k$, we define
	$$\beta_{k,i} := \sum_{j=0}^{\infty} p^{-1}i_j^{(k,i)}Q^{-j}R_k s_j^{(k,i)},$$
	where $s_j^{(k,i)} \in S$ and $i_j^{(k,i)} \in \{-1,1\}$. Lemma \ref{lem5.2}(i) guarantees that the expansion of $\beta_{k,i}$ is unique. By combining this with the definitions of $\Delta_{k,1}$ and $\Delta_{k,2}$, we may choose two positive integers $K$ and $H$ that satisfy the following two conditions:
\begin{enumerate}
	\item[\rm(1)] $K=0$ when $\cup_{k=1}^{N}\Delta_{k,1}=\emptyset$.  Otherwise, for any $1\le k\le N$, if $\Delta_{k,1}\neq\emptyset $, for any $1\le i\le t_{k}$, there exists $ 1\le l\le K$ such that $i_{l}^{(k,i)}|s_{l}^{(k,i)}|>0$.
	\item[\rm(2)] $H=0$ when $\cup_{k=1}^{N}\Delta_{k,2}=\emptyset$.  Otherwise, for any $1\le k\le N$, if $\Delta_{k,2}\neq\emptyset$, for any $t_{k}+1\le i\le t_{k}+h_{k}$, there exists $ H+1\le l\le H+K$ such that $i_{l}^{(k,i)}|s_{l}^{(k,i)}|<0$.
\end{enumerate}
For each non-negative integer \( c \), we define the sequence \( \{\tilde{i}_m\}_{m \in \mathbb{N}} \) by setting
\[
\tilde{i}_{c(K+H+1)} = \tilde{i}_{c(K+H+1)+1} = \dots = \tilde{i}_{c(K+H+1)+K} = -1,
\]
and
\[
\tilde{i}_{c(K+H+1)+K+1} = \tilde{i}_{c(K+H+1)+K+2} = \dots = \tilde{i}_{c(K+H+1)+K+H} = 1.
\]
Furthermore, we choose \( \tilde{\mathbf{I}} = \tilde{i}_0 \tilde{i}_1 \dots \tilde{i}_{K+H} \tilde{i}_{K+H+1} \dots \), define
\begin{equation*}
	S_{K,H,R_k} = \sum_{j=1}^{H} p^{-1}Q^j R_k S - \sum_{j=H+1}^{H+K+1} p^{-1}Q^j R_k S
\end{equation*}
and \( D_{K+H+1} = \sum_{j=0}^{H+K} Q^j D \). From the definitions of $K$ and $H$, it follows immediately that
\begin{equation*}
	T\left(Q^{K+H+1}, S_{K,H,R_k}\right) = \left\{ p^{-1}\sum_{j=0}^{\infty}\tilde{i}_j Q^{-j}R_k s_j : s_j \in S \right\}
\end{equation*}
and \( T(Q^{K+H+1}, S_{K,H,R_k}) \cap \mathcal{Z}(\widehat{\mu}_{Q,D}) = \emptyset \). By adapting the proof technique of Lemma \ref{lem5.1}, we establish that for all $1 \le k \le N$, $\Lambda(Q^{K+H+1}, S_{K,H,R_k})$ is a spectrum of $\mu_{Q^{K+H+1}, D_{K+H+1}}$. From the definition of $S_{K,H,R_k}$, there exists an infinite word $\mathbf{I} = i_0 i_1 \dots \in \{-1,1\}^{\mathbb{N}}$ such that $\Lambda(Q, p^{-1}\mathbf{I}QR_k S) = \Lambda(Q^{K+H+1}, S_{K,H,R_k})$. Since $R_k Q = Q R_k$ and by Lemma \ref{lem5.1}, we infer that $\Lambda(Q, p^{-1}\mathbf{I}QR_k S) = R_k \Lambda(Q, p^{-1}\mathbf{I}QS)$ and that $\Lambda(Q, p^{-1}\mathbf{I}QS)$ is a spectrum of $\mu_{Q,D}$. Note that $\mu_{Q^{K+H+1}, D_{K+H+1}} = \mu_{Q,D}$, which implies that $R_k \Lambda(Q, p^{-1}\mathbf{I}QS)$ is a spectrum of $\mu_{Q,D}$ for all $1 \le k \le N$.

In summary, the proof of Lemma \ref{lem5.3} is now complete.
\end{proof}
With the above technical lemma, we proceed to prove Theorem \ref{thm1.4}.
\begin{proof}[\textbf{Proof of Theorem \ref{thm1.4}}]
$\Rightarrow$. Since $R = \mathrm{diag}[\frac{p_1}{q_1}, \frac{p_2}{q_2}, \dots, \frac{p_n}{q_n}] \in M_n(\frac{\mathbb{Z}\setminus p\mathbb{Z}}{\mathbb{Z}\setminus p\mathbb{Z}})$, there exist integers $k_2, k_3, \dots, k_n \in \{1, 2, \dots, p-1\}$ such that $q_1 \equiv k_i q_i \pmod{p}$ for each $i = 2, 3, \dots, n$. Define
\begin{equation*}
	R_{1}=\diag[p_{1},k_{2}p_{2},k_{3}p_{3},\dots,k_{n}p_{n}] \;\text{and}\;R_{2}=\diag[q_{1},k_{2}q_{2},k_{3}q_{3},\dots,k_{n}q_{n}],
\end{equation*}
	so that $R = R_1 R_2^{-1}$. It is straightforward to verify that $R_2\bm{a} = q_1\bm{a} \pmod{p\mathbb{Z}^n}$. 	Since there exists $k \in \{1, \dots, p-1\}$ such that $\mathrm{diag}[p_1, \dots, p_n]\bm{a} = k \mathrm{diag}[q_1, \dots, q_n]\bm{a} \pmod{p\mathbb{Z}^n}$, we deduce that there exists $l \in \{1, \dots, p-1\}$ such that $R_1\bm{a} = l\bm{a} \pmod{p\mathbb{Z}^n}$. By Lemma \ref{lem5.3}, there exists a spectrum $\Lambda$ of $\mu_{Q,D}$ for which $R_1\Lambda$ and $R_2\Lambda$ are spectra of $\mu_{Q,D}$. Consequently, $R(R_2\Lambda) = R_1\Lambda$, which is additionally a spectrum of $\mu_{Q,D}$. We thus conclude that $R$ is a spectral eigenmatrix of $\mu_{Q,D}$.

$\Leftarrow$. Suppose $\Lambda$ and $R\Lambda$ are spectra of $\mu_{Q,D}$; then we have the decomposition:
\begin{equation*}
	\Lambda = \bigcup_{s\in S} Q\left(\dfrac{s}{p} + \Omega_s\right) \; \text{and} \; R\Lambda = \bigcup_{s\in S} Q\left(\dfrac{s}{p} + \Omega_s'\right).
\end{equation*}
Moreover, we derive
\begin{equation}\label{eq5.4}
	R\bigcup_{s\in S}(s + p\Omega_s) = \bigcup_{s\in S}\left(s + p\Omega_s'\right) \subset \mathbb{Z}^n,
\end{equation}
which implies $R = \mathrm{diag}[\frac{p_1}{q_1}, \dots, \frac{p_n}{q_n}] \in M_n(\mathbb{Q})$. We continue to verify that $R \in M_n(\frac{\mathbb{Z}\setminus p\mathbb{Z}}{\mathbb{Z}\setminus p\mathbb{Z}})$. Assume for contradiction that this is not the case; then there exist $i,j \in \{1, 2, \dots, n\}$ such that $p \mid p_i$ and $p \mid q_j$. Without loss of generality, let $i = j = 1$ (all other cases follow by analogous arguments). If \( p \mid p_1 \), then
\[
R \bigcup_{s \in S} ( s + p \Omega_s ) = \bigcup_{s \in S} ( s + p \Omega_s' ) \subset p\mathbb{Z} \times \mathbb{Z}^{n-1},
\]
a contradiction. If \( p \mid q_1 \), then
\[
R \bigcup_{s \in S \setminus \{\bm{0}\}} (s + p\Omega_s) \nsubseteq \mathbb{Z}^n,
\]
another contradiction. Therefore, $R \in M_n(\frac{\mathbb{Z}\setminus p\mathbb{Z}}{\mathbb{Z}\setminus p\mathbb{Z}})$. By \eqref{eq5.4} and the definition of $S$, there exists $k \in \{1, 2, \dots, p-1\}$ such that $\mathrm{diag}[p_1, \dots, p_n]\bm{a} = k\mathrm{diag}[q_1, \dots, q_n]\bm{a} \pmod{p\mathbb{Z}^n}$.

The proof of the theorem is now complete.
\end{proof}
Finally, we present several examples to illustrate our main theorem.
\begin{exam}\label{exam5.4}
Let \( \mu_{Q,D} \) be a self-affine measure defined by \( \eqref{1.1} \), where \( Q \) is an expansion matrix and \( D =\{(0,0)^{t},(1,0)^{t},(1,1)^{t}\} \). We have the following conclusions.
\begin{enumerate}
	\item[\rm(i)]If $Q=\diag[3q,3q]$ with $q>1$, then $R\in M_{2}(\mathbb{Z}) $ is the second type
	spectral eigenmatrix of $\mu_{Q,D}$ if and only if there exists $k\in\{1,2\} $ such that $R(1,1)^{t}=k(1,1)^{t}\pmod {3\mathbb{Z}^{2}}.$
	\item[\rm(ii)]If $Q=\diag[3l_{1},3l_{2}]$ with $|l_{1}l_{2}|>1$, then diagonal matrix $R\in M_{2}(\mathbb{R}) $ is the second type spectral eigenmatrix of $\mu_{Q,D}$ if and only if $R=\diag[\frac{p_{1}}{q_{1}},\frac{p_{2}}{q_{2}}]\in M_{2}(\frac{\mathbb{Z}\backslash 3\mathbb{Z}}{\mathbb{Z}\backslash 3\mathbb{Z}})$ and there exists $k\in\{1,2\} $ such that $\diag[p_{1},p_{2}](1,1)^{t}=k\diag[q_{1},q_{2}](1,1)^{t}\pmod {3\mathbb{Z}^{2}}.$
\end{enumerate}
\end{exam}
\begin{proof}
It can be easily calculated that
$\mathcal{Z}(\widehat{\delta}_{D})=(\frac{1}{3}(1,1)^{t}+\mathbb{Z}^{2})\cup(\frac{2}{3}(1,1)^{t}+\mathbb{Z}^{2})$, which satisfies \eqref{1.3}.
Therefore, we can directly obtain $(i)$ and $(ii)$ from Theorem \ref{thm1.3} and Theorem \ref{thm1.4}, respectively.
\end{proof}
\begin{exam}
	If \( Q=\diag[6,3]  \) in Example \ref{exam5.4}, then
	\begin{equation*}
		\Lambda=\sum_{i=1}^{k}Q^{i-1}\left\{
		{\left({\begin{array}{*{20}{c}}
					0\\
					0\\
			\end{array}}\right)},
		{\left({\begin{array}{*{20}{c}}
					2\\
					
					1\\
			\end{array}}\right)},
		{\left({\begin{array}{*{20}{c}}
					-2\\
					-1\\
			\end{array}}\right)}  \right\},\;\; k\in\mathbb{N}
	\end{equation*}
is a spectrum of \( \mu_{Q,D}\). Let  $R=\diag[\frac{1}{2},\frac{1}{5}]$, then $\diag[1,1](1,1)^{t}=2\diag[2,5](1,1)^{t}\pmod {3\mathbb{Z}^{2}}.$ Similar to the proof of Lemma \ref{lem5.3}, we can obtain that $\diag[2,5]\Lambda$ is also a spectrum of $\mu_{Q,D}$. Thus, we get that $R=\diag[\frac{1}{2},\frac{1}{5}]$ is the second type
	spectral eigenmatrix of $\mu_{Q,D}$.
\end{exam}

\end{document}